\documentclass[11 pt]{amsart}
\usepackage{amsmath}
\usepackage{latexsym,amsfonts,amssymb,mathrsfs}
\usepackage{mathdots}
\usepackage{color}
\usepackage [all,2cell,color]{xy}
\usepackage{enumitem} 
\usepackage{ dsfont } 

\usepackage{bm} 

\usepackage{tikz}
\usetikzlibrary{cd}
\usetikzlibrary{calc}
\usetikzlibrary{math}
\usetikzlibrary{arrows}
\usetikzlibrary{fit}
\usetikzlibrary{positioning}
\usetikzlibrary{decorations.pathmorphing}
\usetikzlibrary{decorations.pathreplacing}
\usetikzlibrary{ decorations.markings} 
\tikzset{Rightarrow/.style={double equal sign distance,>={Implies},->},
triple/.style={-,preaction={draw,Rightarrow}},
quadruple/.style={preaction={draw,Rightarrow,shorten >=0pt},shorten >=1pt,-,double,double
distance=0.2pt}}

\pgfarrowsdeclarecombine{twotriang}{twotriang}{stealth}{stealth}%
{stealth}{stealth}
\tikzset{mono/.style={>-stealth}} 
\tikzset{epi/.style={-twotriang}} 
\tikzset{twoarrowlonger/.style={double,double distance=1.5pt,
shorten <=5pt,shorten >=6pt,
decoration={markings,mark=at position -4pt with {\arrow[scale=1.75]{>}}},
preaction={decorate}}}

\usepackage{mathtools} 

\usepackage[
textwidth=3cm,
textsize=small,
colorinlistoftodos]
{todonotes}

\usepackage[twoside=false,left=4cm,right=4cm,top=3cm,bottom=3cm]{geometry}

\theoremstyle{plain}   
\newtheorem{thm}{Theorem}[section] 
\makeatletter\let\c@thm\c@thm\makeatother

\makeatletter\let\c@cor\c@thm\makeatother
\newtheorem{lem}{Lemma}[section]
\makeatletter\let\c@lem\c@thm\makeatother
\newtheorem{prop}{Proposition}[section]
\makeatletter\let\c@prop\c@thm\makeatother

\makeatletter\let\c@claim\c@thm\makeatother

\makeatletter\let\c@question\c@thm\makeatother

\newtheorem*{unnumberedtheorem}{Theorem}

\theoremstyle{definition}

\newtheorem{defn}{Definition}[section]
\makeatletter\let\c@defn\c@thm\makeatother
\newtheorem{const}{Construction}[section]
\makeatletter\let\c@const\c@thm\makeatother
\newtheorem{notn}{Notation}[section]
\makeatletter\let\c@notn\c@thm\makeatother

\theoremstyle{remark}

\newtheorem{rmk}{Remark}[section]
\makeatletter\let\c@rmk\c@thm\makeatother
\newtheorem{ex}{Example}[section]
\makeatletter\let\c@ex\c@thm\makeatother

\makeatletter\let\c@observationn\c@thm\makeatother

\makeatletter\let\c@digression\c@thm\makeatother

\makeatletter
\let\c@equation\c@thm
\numberwithin{equation}{section}
\makeatother

\usepackage{hyperref}
\usepackage[capitalise]{cleveref}

\newcommand{\newrefformat}[2]{}

\crefname{lem}{Lemma}{Lemmas}
\crefname{thm}{Theorem}{Theorems}
\crefname{defn}{Definition}{Definitions}
\crefname{notn}{Notation}{Notations}
\crefname{const}{Construction}{Constructions}
\crefname{prop}{Proposition}{Propositions}
\crefname{rmk}{Remark}{Remarks}
\crefname{cor}{Corollary}{Corollaries}
\crefname{equation}{Display}{Displays}
\crefname{ex}{Example}{Examples}

\newcommand{\cA}{\mathcal{A}}
\newcommand{\cB}{\mathcal{B}}

\newcommand{\cD}{\mathcal{D}}
\newcommand{\cE}{\mathcal{E}}

\newcommand{\cP}{\mathcal{P}}

\newcommand{\cS}{\mathcal{S}}

\newcommand{\set}{\cS\!\mathit{et}}

\newcommand{\sset}{\mathit{s}\set}

\DeclareMathOperator{\id}{id}

\DeclareMathOperator{\Hom}{Hom}
\DeclareMathOperator{\Map}{Map}

\DeclareFontFamily{OT1}{pzc}{}
\DeclareFontShape{OT1}{pzc}{m}{it}{<-> s * [1.10] pzcmi7t}{}
\DeclareMathAlphabet{\mathpzc}{OT1}{pzc}{m}{it}

\usepackage{pigpen}
\newcommand{\po}{\ar@{}[dr]|{\text{\pigpenfont R}}}
\newcommand{\pb}{\ar@{}[dr]|{\text{\pigpenfont J}}}

\tikzset{arrow/.style={-stealth}} 
\tikzset{arrowshorter/.style={-stealth, shorten <=2pt, shorten >=2pt}}

\title{The $S_\bullet$-construction as an equivalence\\
between $2$-Segal spaces and\\
stable augmented double Segal spaces}

\author{Martina Rovelli}
\address{Department of Mathematics and Statistics, University of Massachusetts Amherst, Amherst, USA}
\email{mrovelli@umass.edu} 

\begin{document}

\maketitle

\begin{abstract}

This note is a contribution for a proceedings volume of the workshop \emph{Higher Segal Spaces and their Applications to Algebraic K-Theory, Hall Algebras, and Combinatorics}. The content is a streamlined exposition based on a talk about a result by Bergner-Osorno-Ozornova-Rovelli-Scheimbauer from WITII. We discuss how a generalized version of Waldhausen's $S_\bullet$-construction describes a correspondence between $2$-Segal spaces and certain double Segal spaces, which satisfy further conditions of stability and augmentation. 
\end{abstract}

\tableofcontents

\section*{Introduction}

As the equivalent notions of a \emph{$2$-Segal space} \cite{DKbook} and of a \emph{decomposition space} \cite{GCKT1} arose, the most prominent example which was identified was that of \emph{Waldhausen's $S_\bullet$-construction} \cite{waldhausen} of categorical structures of a homological algebraic flavor. This was treated for Abelian categories in \cite[\textsection10]{GCKT1}, for (proto-)exact categories in \cite[\textsection2.4]{DKbook}, for exact and stable $\infty$-categories in \cite[\textsection7]{DKbook}.

Two natural questions arose:
\begin{enumerate}[leftmargin=*]
    \item Do all examples of $2$-Segal spaces arise as the $S_\bullet$-construction of a suitably generalized input?
    \item Does the $S_\bullet$-construction remember all the structure of the original input?
\end{enumerate}
In other words, it is natural to wonder to which degree the $S_\bullet$-construction is surjective and to which degree it is injective.

These questions were addressed by one of the teams of WITII -- Bergner--Osorno--Ozornova--Rovelli--Scheimbauer -- in \cite{BOORS1,BOORS3}. The team identified an appropriate class of double Segal spaces with further structure, called \emph{stable augmented double Segal spaces}, to which the $S_\bullet$-construction can be generalized, as well as an inverse construction: the \emph{path construction} $\cP$.

In this note, we outline the main ingredients for the proof of a soft version of the correspondence of $2$-Segal spaces and stable augmented double Segal spaces, which we state as \cref{EquivalenceV1}:

\begin{unnumberedtheorem}[Soft version]
The path construction $\cP$ and (a variant of) the $S_\bullet$-construction define inverse bijections
    \[
    \cP\colon\mathrm{2SegSp}/_\simeq\ \cong\ \mathrm{saDblSegSp}/_\simeq\colon S_\bullet\circ\widetilde{(-)}
    \]
between $2$-Segal spaces up to equivalence and stable augmented double Segal spaces up to equivalence.
\end{unnumberedtheorem}

We also mention as \cref{EquivalenceV2} the stronger version of the result, which studies the correspondence between $2$-Segal spaces and stable augmented double Segal spaces at the level of $\infty$-categories:

\begin{unnumberedtheorem}[Strong version]
The path construction $\cP$ and (a variant of) the $S_\bullet$-construction define inverse equivalences of $\infty$-categories
\[
\bm{\cP}\colon\bm{\mathrm{2SegSp}}\simeq\bm{\mathrm{saDblSegSp}}\colon \bm{S_\bullet\circ\widetilde{(-)}}
\]
between the $\infty$-category of $2$-Segal spaces and the $\infty$-category of stable augmented double Segal spaces.
\end{unnumberedtheorem}

This note is written with the intention of being as elementary as possible, focusing on conveying intuition rather than describing the most general framework. The information presented is accurate; however, certain technical details have been excluded to favor exposition, referring the reader to the original source(s) for the complete proofs.

\addtocontents{toc}{\protect\setcounter{tocdepth}{1}}
\subsection*{Acknowledgements}

We are grateful to Julie Bergner, Joachim Kock, Mark Penney, and Maru Sarazola for organizing a wondeful workshop at the Banff International Research Station: \emph{Higher Segal Spaces and their Applications to Algebraic K-Theory, Hall Algebras, and Combinatorics} (24w5266, originally 20w5173). We are also thankful to Viktoriya Ozornova for useful feedback and assistance in writing this note, and to Claudia Scheimbauer for creating and letting us use the picture macros from our previous papers.
We acknowledge the support from the National Science Foundation under Grant No. DMS-2203915.

\addtocontents{toc}{\protect\setcounter{tocdepth}{2}}

\section{$2$-Segal spaces}

\label{Section2SegalSpaces}

In this section we recall the notion of a \emph{$2$-Segal space}, which was introduced in \cite{DKbook} and \cite{GCKT1} independently, and note the properties which will be used in the remainder of the paper. We also discuss various structures one can consider on the class of $2$-Segal spaces.

Let $\Delta$ denote the simplex category, and $\sset$ the category of simplicial sets, which we refer to as \emph{spaces}.

\begin{notn}
A \emph{simplicial space} is a functor $X\colon\Delta^{\mathrm{op}}\to\sset$. Explicitly, it consists of a space $X_n$ of \emph{$n$-simplices} for $n\geq0$, together with the following structure maps:
\begin{itemize}[leftmargin=*]
    \item the $i$-th \emph{face map} $d_i\colon X_n\to X_{n-1}$ for $n>0$ and $0\leq i\leq n$; and
    \item the $i$-th \emph{degeneracy map} $s_i\colon X_n\to X_{n+1}$ for $n\geq0$ and $0\leq i\leq n$.
\end{itemize}
These structure maps are subject to the usual
simplicial identities.
\end{notn}

In this note we take the following as the definition of a \emph{$2$-Segal space}. It was shown as \cite[Proposition 1.18]{BOORS3}
that it is equivalent to the original definition from \cite[Definition~2.3.1]{DKbook}.

\begin{defn}
\label{2SegalSpace}
    A simplicial space $X$ is a \emph{$2$-Segal space} if for all $n\geq2$ the canonical maps induce equivalences of spaces
\[
X_{n+1}\xrightarrow{\simeq} X_{\{0,1,2\}}\times^h_{X_{\{0,2\}}}X_{\{0,2,\dots,n+1\}}\text{ and } X_{n+1}\xrightarrow{\simeq} X_{\{n-2,n-1,n\}}\times^h_{X_{\{n-2,n\}}}X_{\{0,\dots,n-2,n\}}
\]
induced, respectively, by the commutative squares of spaces
\[
\begin{tikzcd}
X_{n+1}\arrow[r,""]\arrow[d,"" swap]&X_{\{0,2,\dots,n+1\}}\arrow[d,""]\\
X_{\{0,1,2\}}\arrow[r,"" swap]&X_{\{0,2\}}
\end{tikzcd}
\text{ and }
\begin{tikzcd}
X_{n+1}\arrow[r,""]\arrow[d,"" swap]&X_{\{0,\dots,n-2,n\}}\arrow[d,""]\\
X_{\{n-2,n-1,n\}}\arrow[r,"" swap]&X_{\{n-2,n\}}
\end{tikzcd}
\]
\end{defn}

\begin{rmk}
\label{unital}
It was shown in \cite{FGKPW} that every $2$-Segal space $X$ is \emph{unital} in the sense of \cite[Definition~2.5.2]{DKbook} (equivalently, every $2$-Segal space $X$ is a \emph{decomposition space} in the sense of \cite[Definition~3.1]{GCKT1}). This guarantees in particular
that there are equivalences of spaces
\[
(s_1,d_0)\colon X_1\xrightarrow{\simeq} X_2\times^h_{X_1}X_0\quad\text{ and }\quad (s_0,d_1)\colon X_1\xrightarrow{\simeq} X_2\times^h_{X_1}X_0
\]
induced, respectively, by the commutative squares
\[
\begin{tikzcd}
X_{1}\arrow[r,"d_0"]\arrow[d,"s_1" swap]&X_0\arrow[d,"s_0"]\\
X_2\arrow[r,"d_0" swap]&X_1
\end{tikzcd}
\quad\text{ and }\quad
\begin{tikzcd}
X_{1}\arrow[r,"d_1"]\arrow[d,"s_0" swap]&X_0\arrow[d,"s_0"]\\
X_2\arrow[r,"d_2" swap]&X_1
\end{tikzcd}
\]
\end{rmk}

The main result of this note will explain how there is an appropriate correspondence between the notion of a $2$-Segal space and that of a certain double Segal space, from which a variety of examples will arise.

To make this correspondence precise, we shall identify the appropriate categorical structure that $2$-Segal spaces assemble into. As a first approximation, one can consider an ordinary category:

\begin{rmk}
There is a category $\mathrm{2SegSp}$, in which:
\begin{itemize}[leftmargin=*]
    \item the objects are the $2$-Segal spaces;
    \item the morphisms are the maps of simplicial spaces between $2$-Segal spaces;
    \item the composition and identity operators are inherited from those in the category of simplicial spaces.
\end{itemize}
\end{rmk}

However, the category $\mathrm{2SegSp}$ of $2$-Segal spaces fails at encompassing the meaningful homotopy theory of such, and would in particular be insufficient to state the desired classification result.
One way to encode the homotopy theory of $2$-Segal spaces is to first define an appropriate notion of equivalence between $2$-Segal spaces:

\begin{defn}
Let $X$ and $Y$ be $2$-Segal spaces. A map of simplicial spaces $f\colon X\to Y$ is an \emph{equivalence} of $2$-Segal spaces if for all $n\geq0$ it defines an equivalence of spaces $f_n\colon X_n\xrightarrow{\simeq} Y_n$.
\end{defn}

With this notion of equivalence between $2$-Segal spaces, one can consider the set of equivalence classes of $2$-Segal spaces:

\begin{rmk}
Consider the relation on the class of $2$-Segal spaces given by declaring that $X$ is related to $Y$ if there exists an equivalence $f\colon X\xrightarrow{\simeq}Y$ of $2$-Segal spaces. It follows from the definitions that this relation is reflexive
and transitive, though without imposing further technical conditions it is not a priori symmetric.
Nevertheless, one can consider the set $\mathrm{2SegSp}/_\simeq$ of $2$-Segal spaces up to the equivalence relation it generates.
\end{rmk}

We will see as \cref{EquivalenceV1} that the set $\mathrm{2SegSp}/_\simeq$ of equivalence classes can be used to phrase an elementary version of the classification theorem. However, to state the most comprehensive version of the result, which we will mention as \cref{EquivalenceV2}, one should enhance the structure even further. Precisely, one should consider the $\infty$-category of $2$-Segal spaces, in which equivalences are formally inverted in an appropriate sense.

One way to incarnate the notion of an \emph{$\infty$-category} is in the form of a simplicial category. Recall that a \emph{simplicial category} is a category enriched over the category of simplicial sets. It consists in particular of a set of objects, and a mapping space between any two pairs of objects, together with identity and composition simplicial operators which satisfy associativity and unitality axioms.

Given a category with a class of equivalences, one can construct its simplicial localization in the form of a simplicial category. Various implementations of this procedure are available, and one of them is the \emph{hammock localization} from \cite{DwyerKanCalculating} (cf.~also \cite{Hinich}).

By simplicially localizing the category of $2$-Segal spaces at the class of equivalences of $2$-Segal spaces, we obtain a simplicial category in which all equivalences of $2$-Segal spaces are universally inverted in an appropriate weak sense:

\begin{const}
\label{2SegSp}
    The simplicial localization $\bm{\mathrm{2SegSp}}$ of the category of $2$-Segal spaces and equivalences $2$-Segal spaces
is a simplicial category in which:
    \begin{itemize}[leftmargin=*]
        \item The set of objects is given by the collection of the $2$-Segal spaces;
        \item Given $2$-Segal spaces $X$ and $Y$, the mapping space of maps from a $X$ to $Y$ is given by the space of hammocks from $X$ to $Y$.
        \item Composition and identity operators are given by concatenating hammocks, and taking the trivial hammock, respectively.
    \end{itemize}
\end{const}


\section{Stable augmented double Segal spaces}

In this section we discuss the notion of a \emph{stable augmented double Segal space} from \cite{BOORS3}, and note the properties which will be used in the remainder of the paper. We also discuss various structures one can consider on the class of stable augmented double Segal spaces.

Denote by $\Sigma$ the category obtained by freely adding a terminal object to $\Delta \times \Delta$, which could be visualized as follows:
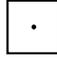
\[
\begin{tikzcd}[scale=0.7, every label/.append style = {font = \tiny}]
{[-1]}  \\
\mbox{} & 
(0,0)
\arrow{lu}{} 
\arrow[r, arrow, shift left=1ex] \arrow[r, arrow, shift right=1ex] 
\arrow[d, arrow, shift left=1ex] \arrow[d, arrow, shift right=1ex]& 
(0,1)\arrow[l, arrow]  \arrow[r, arrow] \arrow[r, arrow, shift left=1.5ex] \arrow[r, arrow, shift right=1.5ex]
\arrow[d, arrow, shift left=1ex] \arrow[d, arrow, shift right=1ex]& 
(0,2) 
\arrow[l, arrowshorter, shift left=0.75ex] \arrow[l, arrowshorter, shift right=0.75ex] 
\arrow[d, arrow, shift left=1ex] \arrow[d, arrow, shift right=1ex]&[-0.8cm]\cdots\\
\mbox{} & 
(1,0) \arrow[r, arrow, shift left=1ex] \arrow[r, arrow, shift right=1ex]
\arrow[u, arrowshorter]
\arrow[d, arrow] \arrow[d, arrow, shift left=1.5ex] \arrow[d, arrow, shift right=1.5ex]
 & 
(1,1)  \arrow[l, arrow]  \arrow[r, arrow] \arrow[r, arrow, shift left=1.5ex] \arrow[r, arrow, shift right=1.5ex] 
\arrow[u, arrowshorter]
\arrow[d, arrow] \arrow[d, arrow, shift left=1.5ex] \arrow[d, arrow, shift right=1.5ex]& 
(1,2) \arrow[l, arrowshorter, shift left=0.75ex] \arrow[l, arrowshorter, shift right=0.75ex] 
\arrow[u, arrowshorter]
\arrow[d, arrow] \arrow[d, arrow, shift left=1.5ex] \arrow[d, arrow, shift right=1.5ex]&\cdots\\
\mbox{} & 
(2,0) \arrow[r, arrow, shift left=1ex] \arrow[r, arrow, shift right=1ex]
\arrow[u, arrowshorter, shift left=0.75ex] \arrow[u, arrowshorter, shift right=0.75ex] 
& 
(2,1)  \arrow[l, arrow]  \arrow[r, arrow] \arrow[r, arrow, shift left=1.5ex] \arrow[r, arrow, shift right=1.5ex]
\arrow[u, arrowshorter, shift left=0.75ex] \arrow[u, arrowshorter, shift right=0.75ex] & 
(2,2) \arrow[l, arrowshorter, shift left=0.75ex] \arrow[l, arrowshorter, shift right=0.75ex] 
\arrow[u, arrowshorter, shift left=0.75ex] \arrow[u, arrowshorter, shift right=0.75ex] &
\cdots\\[-0.7cm]
\mbox{} & \vdots & \vdots & \vdots &\ddots\\
\end{tikzcd}
\]

\begin{notn}
A \emph{preaugmented bisimplicial space} is a functor $\cD\colon\Sigma^{\mathrm{op}}\to\sset$. Precisely, it consists of a space $\cD_{a,b}$ of \emph{$(a,b)$-bisimplices} for $a,b\geq0$ and an \emph{augmentation space} $\cD_{-1}$, together with the following structure maps:
\begin{itemize}[leftmargin=*]
        \item the \emph{vertical $i$-th face map} $d^v_i\colon\cD_{n,\ell}\to\cD_{n-1,\ell}$ for $\ell\geq0$ and $n>0$,
        \item the \emph{vertical $i$-th degeneracy map} $s^v_i\colon\cD_{n,\ell}\to\cD_{n+1,\ell}$ for $\ell\geq0$ and $n\geq0$,
        \item the \emph{horizontal $i$-th face map} $d^h_i\colon\cD_{\ell,n}\to\cD_{\ell,n-1}$ for $\ell\geq0$ and $n>0$,
        \item the \emph{horizontal $i$-th degeneracy map} $s^h_i\colon\cD_{\ell,n}\to\cD_{\ell,n+1}$ for $\ell\geq0$ and $n\geq0$,
        \item the \emph{augmentation map} $\varepsilon\colon\cD_{-1}\to\cD_{0,0}$.
\end{itemize}
Face and degeneracy maps are subject to the usual bisimplicial identities, and the augmentation is not subject to any relation.
For further reference, we fix the notation for:
\begin{itemize}[leftmargin=*]
\item the \emph{vertical iterated top face map}
$d^v_\top\coloneqq d_n^v\circ d_{n-1}^v\circ\dots \circ d_{k+1}^v\colon\cD_{n,\ell}\to\cD_{k,\ell}$ for $\ell\geq0$ and $n\geq k\geq0$;
\item the \emph{horizontal iterated top face map}
$d^h_\top\coloneqq d_n^h\circ d_{n-1}^h\circ\dots \circ d_{k+1}^h\colon\cD_{\ell,n}\to\cD_{\ell,k}$ for $\ell\geq0$ and $n\geq k\geq0$;
\item the \emph{vertical iterated bottom face map}
$d^v_\bot\coloneqq d_0^v\circ d_0^v\circ\dots \circ d_0^v\colon\cD_{n,\ell}\to\cD_{k,\ell}$ for $\ell\geq0$ and $n\geq k\geq0$;
\item the \emph{horizontal iterated bottom face map}
$d^h_\bot\coloneqq d_0^h\circ d_0^h\circ\dots \circ d_0^h\colon\cD_{\ell,n}\to\cD_{\ell,k}$ for $\ell\geq0$ and $n\geq k\geq0$;
\end{itemize}
\end{notn}

\begin{rmk}
\label{Interpretation}
 If $\cD$ is a preaugmented bisimplicial space, we'd like to enforce the interpretation that:
    \begin{itemize}[leftmargin=*]
        \item $\cD_{0,0}$ is a space of objects, which we generically depict as
\begin{center}
    \begin{tikzpicture}[scale=0.7]
        \draw[thick] (6.5,0.5) rectangle (7.5,-0.5);
        \begin{scope}[xshift=4.5cm, yshift=-0.5cm]
            \draw[fill] (2.5, 0.5) circle (1pt) node(a12){};
            \end{scope}
    \end{tikzpicture}
\end{center}
        \item $\cD_{-1}$ is a subspace of distinguished objects (cf.~\cref{AugmentatonInjective}), which we generically depict as
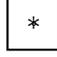
\begin{center}
    \begin{tikzpicture}[scale=0.7]
        \draw[thick] (6.5,0.5) rectangle (7.5,-0.5);
        \begin{scope}[xshift=4.5cm, yshift=-0.5cm]
            \draw (2.5, 0.5) circle (1pt) node(a12){$*$};
            \end{scope}
    \end{tikzpicture}
\end{center}
        \item $\cD_{1,0}$ and $\cD_{0,1}$ are, respectively, a space of vertical morphisms and a space of horizontal morphisms, which we generically depict as
 
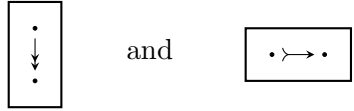
\begin{center}
    \begin{tikzpicture}[scale=0.7]
        \draw[thick] (3,0.5) rectangle (4,-1.5);
        \begin{scope}[xshift=3.5cm, yshift=-0.5cm]
            \draw[fill] (0, -0.5) circle (1pt) node(a12){};
            \draw[fill] (0, 0.5) circle (1pt) node(a02){};
            \draw[epi] (a02)--(a12);
        \end{scope}
        \draw (6.5, 0) node[anchor=north east] (j5){\quad and \quad};
        \draw[thick] (7.5,0) rectangle (9.5,-1);
        \begin{scope}[xshift=5cm, yshift=-0.5cm]
            \draw[fill] (3, 0) circle (1pt) node(a02){};
            \draw[fill] (4, 0) circle (1pt) node (a03){};
            \draw[mono] (a02)--(a03);
        \end{scope}
    \end{tikzpicture}
\end{center}
        \item $\cD_{1,1}$ is a space of squares, which we picture as follows
 
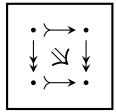
\begin{center}
    \begin{tikzpicture}[scale=0.7]
        \draw[thick] (6.5,0.5) rectangle (8.5,-1.5);
        \begin{scope}[xshift=4cm, yshift=-0cm]
            \draw[fill] (3, -1) circle (1pt) node(a12){};
            \draw[fill] (3, 0) circle (1pt) node(a02){};
            \draw[fill] (4, 0) circle (1pt) node (a03){};
            \draw[fill] (4, -1) circle (1pt) node (a13){};
            \draw[mono] (a02)--(a03);
            \draw[mono] (a12)--(a13);
            \draw[epi] (a02)--(a12);
            \draw[epi] (a03)--(a13);
            \begin{scope}[yshift=-0.3cm]
                \draw[twoarrowlonger] (3.2,0.1)--(3.8,-0.5);
            \end{scope}
        \end{scope}
    \end{tikzpicture}
\end{center}
        \item $\cD_{1,2}$ and $\cD_{2,1}$ are, respectively, spaces of horizontally composable pairs of squares with specified composites and horizontally composable pairs of squares with specified composites, which we picture as follows
 
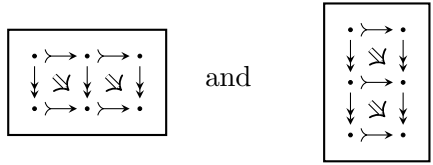
\begin{center}
    \begin{tikzpicture}[scale=0.7]
        \draw[thick] (1.5,-0) rectangle (4.5,-2);
        \begin{scope}[xshift=-1cm, yshift=-0.5cm]
            \draw[fill] (3, -1) circle (1pt) node(a12){};
            \draw[fill] (3, 0) circle (1pt) node(a02){};
            \draw[fill] (4, 0) circle (1pt) node (a03){};
            \draw[fill] (4, -1) circle (1pt) node (a13){};
             \draw[fill] (5, 0) circle (1pt) node (a04){};
            \draw[fill] (5, -1) circle (1pt) node (a14){};
            \draw[mono] (a02)--(a03);
            \draw[mono] (a12)--(a13);
            \draw[mono] (a13)--(a14);
            \draw[mono] (a03)--(a04);
            \draw[epi] (a02)--(a12);
            \draw[epi] (a03)--(a13);
            \draw[epi] (a04)--(a14);
            \begin{scope}[yshift=-0.3cm]
                \draw[twoarrowlonger] (3.2,0.1)--(3.8,-0.5);
                \draw[twoarrowlonger] (4.2,0.1)--(4.8,-0.5);
            \end{scope}
        \end{scope}
        \draw (6.5, -0.5) node[anchor=north east] (j5){\quad and \quad};
        \draw[thick] (7.5,0.5) rectangle (9.5,-2.5);
        \begin{scope}[xshift=5cm, yshift=-0cm]
            \draw[fill] (3, -1) circle (1pt) node(a12){};
            \draw[fill] (3, 0) circle (1pt) node(a02){};
            \draw[fill] (4, 0) circle (1pt) node (a03){};
            \draw[fill] (4, -1) circle (1pt) node (a13){};
            \draw[fill] (3, -2) circle (1pt) node (a04){};
            \draw[fill] (4, -2) circle (1pt) node (a14){};
            \draw[mono] (a02)--(a03);
            \draw[mono] (a12)--(a13);
            \draw[epi] (a13)--(a14);
            \draw[epi] (a03)--(a13);
            \draw[epi] (a02)--(a12);
            \draw[epi] (a12)--(a04);
            \draw[mono](a04)--(a14);
            \begin{scope}[yshift=-0.3cm]
                \draw[twoarrowlonger] (3.2,0.1)--(3.8,-0.5);
                \draw[twoarrowlonger] (3.2,-0.9)--(3.8,-1.5);
            \end{scope}
        \end{scope}
    \end{tikzpicture}
\end{center}
    \end{itemize}
Under this interpretation, we would like to think of
    \begin{itemize}[leftmargin=*]
     \item the augmentation map $\varepsilon\colon\cD_{-1}\to\cD_{0,0}$ as a (homotopy-)inclusion of $\cD_{-1}$ into $\cD_{0,0}$;
     \item the face map $d^v_1\colon\cD_{1,0}\to\cD_{0,0}$ (resp.~$d^h,1\colon\cD_{0,1}\to\cD_{0,0}$)  as the source map for horizontal morphisms (resp.~the source map for vertical morphisms);
 \item the face map $d^v_0\colon\cD_{1,0}\to\cD_{0,0}$ (resp.~$d^h_0\colon\cD_{0,1}\to\cD_{0,0}$)  as the target map for horizontal morphisms (resp.~the target map for vertical morphisms);
    \item the face map $d^v_1\colon\cD_{1,1}\to\cD_{0,1}$ (resp.~$d^h_1\colon\cD_{1,1}\to\cD_{1,0}$) as the horizontal source map for squares (resp.~the vertical source map for squares);
 \item the face map $d^v_0\colon\cD_{1,1}\to\cD_{0,1}$ (resp.~$d^h_0\colon\cD_{1,1}\to\cD_{1,0}$) as the vertical target map for squares (resp.~the horizontal target map for squares);
\item the degeneracy map $s^v_0\colon\cD_{0,0}\to\cD_{1,0}$ (resp.~$s^h_0\colon\cD_{0,0}\to\cD_{0,1}$) as the vertical identity map for objects (resp.~horizontal identity map for objects).
        \end{itemize}
\end{rmk}

In order to enforce the desired behaviour from the various pieces the structure, we focus on the preaugmented bisimplicial spaces which satisfies the following further conditions.

\begin{defn}
\label{sadss}
A preaugmented bisimplicial space $\cD$ is a \emph{stable augmented double Segal space} if the following hold:
\begin{enumerate}[leftmargin=*]
    \item \emph{Double Segality}: For all $k,\ell,m\geq0$, there are equivalences of spaces
    \[
(d^v_{\top},d^v_{\bot})\colon\cD_{k+\ell,m}\xrightarrow{\simeq}\cD_{k,m}\times^h_{\cD_{0,m}}\cD_{\ell,m}
\quad\text{ and }\quad
(d^h_{\top},d^h_{\bot})\colon\cD_{m,k+\ell}\xrightarrow{\simeq}\cD_{m,k}\times^h_{\cD_{m,0}}\cD_{m,\ell}\]
induced, respectively, by the commutative squares of spaces
\[
\begin{tikzcd}
\cD_{k+\ell,m}\arrow[r,"d^v_{\bot}"]\arrow[d,"d^v_{\top}" swap]&\cD_{k,m}\arrow[d,"d^v_{\top}"]\\
\cD_{\ell,m}\arrow[r,"d^v_{\bot}" swap]&\cD_{0,m}
\end{tikzcd}
\quad\text{ and }\quad
\begin{tikzcd}
\cD_{m,k+\ell}\arrow[r,"d^h_{\top}"]\arrow[d,"d^h_{\bot}" swap]&\cD_{m,k}\arrow[d,"d^h_{\bot}"]\\
\cD_{m,\ell}\arrow[r,"d^h_{\top}" swap]&\cD_{m,0}
\end{tikzcd}
\]
    \item \emph{Stability}: There are equivalences of spaces 
     \[
(d^h_1,d^v_1)\colon \cD_{1,1}\xrightarrow{\simeq}\cD_{1,0}\times^h_{\cD_{0,0}}\cD_{0,1}\quad\text{ and }\quad
(d^v_0,d^h_0)\colon \cD_{1,1}\xrightarrow{\simeq}\cD_{0,1}\times^h_{\cD_{0,0}}\cD_{1,0}
     \]
     induced, respectively, by the commutative squares of spaces 
\[
\begin{tikzcd}
\cD_{1,1}\arrow[r,"d^v_{1}"]\arrow[d,"d^h_{1}" swap]&\cD_{0,1}\arrow[d,"d^h_{1}"]\\
\cD_{1,0}\arrow[r,"d^h_{1}" swap]&\cD_{0,0}
\end{tikzcd}
\quad\text{ and }\quad
\begin{tikzcd}
\cD_{1,1}\arrow[r,"d^v_{0}"]\arrow[d,"d^h_{0}" swap]&\cD_{1,0}\arrow[d,"d^h_{0}"]\\
\cD_{0,1}\arrow[r,"d^v_{0}" swap]&\cD_{0,0}
\end{tikzcd}
\]
     \item \emph{Augmentation}: There are equivalences of spaces
     \[
d^v_1\colon \cD_{1,0}\times^h_{\cD_{0,0}}\cD_{-1}\xrightarrow{\simeq}\cD_{0,0}\quad\text{ and }\quad
d^h_0\colon \cD_{-1}\times^h_{\cD_{0,0}}\cD_{0,1}\xrightarrow{\simeq}\cD_{0,0}
     \]
     where the homotopy pullbacks are given by
\[
\begin{tikzcd}
\cD_{1,0}\times^h_{\cD_{0,0}}\cD_{-1}\arrow[rd, phantom, "\lrcorner", very near start]\arrow[r]\arrow[d]&\cD_{1,0}\arrow[d,"d^v_0"]\\
\cD_{-1}\arrow[r,"\varepsilon" swap]&\cD_{0,0}
\end{tikzcd}
\quad\text{ and }\quad
\begin{tikzcd}
\cD_{0,1}\times^h_{\cD_{0,0}}\cD_{-1}\arrow[rd, phantom, "\lrcorner", very near start]\arrow[r]\arrow[d]&\cD_{0,1}\arrow[d,"d^h_1"]\\
\cD_{-1}\arrow[r,"\varepsilon" swap]&\cD_{0,0}
\end{tikzcd}
\]
 \end{enumerate}
\end{defn}

While the structure map $\varepsilon\colon\cD_{-1}\to\cD_{0,0}$ may not be injective, if $\cD$ is a stable augmented double Segal spaces it is at least in a sense a retract up to homotopy. This justifies the idea that $\cD_{-1}$ morally realizes a subspace of $\cD_{0,0}$.

\begin{prop}
\label{AugmentatonInjective}
If $\cD$ is a stable augmented double Segal space, the augmentation map factors as a retract map followed by an equivalence of spaces, as follows:
\[
\begin{tikzcd}
\cD_{-1}\arrow[r,"\varepsilon"]\arrow[rd,hook]&\cD_{0,0}\\
&\cD_{0,1}\times_{\cD_{0,0}}^h\cD_{-1}\arrow[u,"\simeq" swap]
\end{tikzcd}
\]
\end{prop}

\begin{proof}
Consider the homotopy pullback
\[
\begin{tikzcd}
\cD_{1,0}\times^h_{\cD_{0,0}}\cD_{-1}\arrow[rd, phantom, "\lrcorner", very near start]\arrow[r,"\pi_1"]\arrow[d]&\cD_{1,0}\arrow[d,"d^v_0"]\\
\cD_{-1}\arrow[r,"\varepsilon" swap]&\cD_{0,0}
\end{tikzcd}
\]
We observe that there is a commutative diagram of spaces:
\[
\begin{tikzcd}
\cD_{-1}\arrow[d,"\varepsilon" swap]\arrow[r,"s^v_0", dashed]&\cD_{1,0}\times_{\cD_{0,0}}^h\cD_{-1}\arrow[d,"\pi_1" swap]\arrow[rdd,"d_1^v\circ\pi_1", bend left=20]&\\
\cD_{0,0}
\arrow[r,"s^v_0"]\arrow[rrd,"\id" swap, bend right=20]&\cD_{1,0}\arrow[rd,"d^v_1" very near start]&\\
&&\cD_{0,0}
\end{tikzcd}
\]
Here,
\begin{itemize}[leftmargin=*]
    \item the top dashed map
$s_0^v\colon\cD_{-1}\to\cD_{0,1}\times_{\cD_{0,0}}^h\cD_{-1}$
is a retract because
it admits the retraction
$d_0^v\colon\cD_{0,1}\times_{\cD_{0,0}}^h\cD_{-1}\to\cD_{-1}$,
\item the tilted map $d_1^v\circ\pi_1\colon\cD_{1,0}\times_{\cD_{0,0}}^h\cD_{-1}\xrightarrow{\simeq}\cD_{0,0}$ is an equivalence of spaces because of the augmentation condition, and
\item the total composite is the augmentation $\varepsilon\colon\cD_{-1}\to\cD_{0,0}$.
\end{itemize}
The claim then follows. 
\end{proof}

\begin{rmk}
Relying on the intuition laid out in \cref{Interpretation}, the axioms are designed to encode, respectively, the following properties:
    \begin{enumerate}[leftmargin=*]
    \item The double Segality condition demands that the underlying bisimplicial space $\cD|_{\Delta\times\Delta}$ of $\cD$ is a double Segal space in the usual sense. This gives meaning to the fact that horizontal morphisms can be composed, vertical morphisms can be composed, and squares can be composed horizontally and vertically, all in an appropriately associative and unital sense.
 \item The stability conditions demand that each square is determined in an appropriate sense by the span contained in its boundary and also by the cospan contained in its boundary.
    Indeed, the stability conditions give equivalences of spaces
    \[(d^v_1,d^h_1)\colon\cD_{1,1}\xrightarrow{\simeq}\cD_{1,0}\times^h_{\cD_{0,0}}\cD_{0,1}\quad\text{ and }\quad(d^h_0,d^v_0)\colon\cD_{1,1}\xrightarrow{\simeq}\cD_{0,1}\times^h_{\cD_{0,0}}\cD_{1,0}\]
   which can be depicted as
\begin{center}
\begin{tikzpicture}[scale=0.7]
\draw[thick] (-2,0) rectangle (0,-2);
\draw[thick] (0.8,0) rectangle (2.5,-2);
\begin{scope}[xshift=-4.5cm, yshift=-0.5cm]
    \draw[fill] (3, -1) circle (1pt) node(a12){};
    \draw[fill] (3, 0) circle (1pt) node(a02){};
    \draw[fill] (4, 0) circle (1pt) node (a03){};
    \draw[fill] (4, -1) circle (1pt) node (a13){};
    \draw[mono] (a02)--(a03);
    \draw[mono] (a12)--(a13);
    \draw[epi] (a02)--(a12);
    \draw[epi] (a03)--(a13);
    \begin{scope}[yshift=-0.3cm]
        \draw[twoarrowlonger] (3.2,0.1)--(3.8,-0.5);
    \end{scope}
\end{scope}
\draw (0.9, -0.7) node[anchor=north east] (j5){{$\mapsto$}};
\begin{scope}[xshift=-1.8cm, yshift=-0.5cm]
    \draw[fill] (3, -1) circle (1pt) node(a12){};
    \draw[fill] (3, 0) circle (1pt) node(a02){};
    \draw[fill] (4, 0) circle (1pt) node (a03){};
    \draw[mono] (a02)--(a03);
    \draw[epi] (a02)--(a12);
\end{scope}
\draw (4.3, -0.5) node[anchor=north east] (j5){\quad and \quad};
\draw[thick] (4.5,0) rectangle (6.5,-2);
\draw[thick] (7.5,0) rectangle (9.5,-2);
\begin{scope}[xshift=2cm, yshift=-0.5cm]
    \draw[fill] (3, -1) circle (1pt) node(a12){};
    \draw[fill] (3, 0) circle (1pt) node(a02){};
    \draw[fill] (4, 0) circle (1pt) node (a03){};
    \draw[fill] (4, -1) circle (1pt) node (a13){};
    \draw[mono] (a02)--(a03);
    \draw[mono] (a12)--(a13);
    \draw[epi] (a02)--(a12);
    \draw[epi] (a03)--(a13);
    \begin{scope}[yshift=-0.3cm]
        \draw[twoarrowlonger] (3.2,0.1)--(3.8,-0.5);
    \end{scope}
\end{scope}
\draw (7.5, -0.7) node[anchor=north east] (j5){{$\mapsto$}};
\begin{scope}[xshift=5cm, yshift=-0.5cm]
    \draw[fill] (3, -1) circle (1pt) node(a12){};
    \draw[fill] (4, 0) circle (1pt) node (a03){};
    \draw[fill] (4, -1) circle (1pt) node (a13){};
    \draw[mono] (a12)--(a13);
    \draw[epi] (a03)--(a13);
\end{scope}
\end{tikzpicture}
\end{center}
\item The augmentation condition essentially demands that every object receives uniquely from the augmentation space through a horizontal morphism, and it maps uniquely to the augmentation space through a vertical morphism.
Indeed, the two augmentation conditions give equivalences of spaces
    \[d^v_1\colon\cD_{1,0}\times^h_{\cD_{0,0}}\cD_{-1}\xrightarrow{\simeq}\cD_{0,0}\quad\text{ and }\quad
    d^h_0\colon\cD_{-1}\times^h_{\cD_{0,0}}\cD_{0,1}\xrightarrow{\simeq}\cD_{0,0}\]
    which can be depicted as
   \begin{center}
\begin{tikzpicture}[scale=0.7]
\draw[thick] (-1.5,0) rectangle (-0.5,-2);
\begin{scope}[xshift=-4cm, yshift=-0.5cm]
    \draw (3, -1) node(a12){$*$};
    \draw[fill] (3, 0) circle (1pt) node(a02){};
    \draw[epi] (a02)--(a12);
\end{scope}
\draw (0.5, -0.7) node[anchor=north east] (j5){{$\mapsto$}};
\draw[thick] (0.5,0) rectangle (1.5,-2);
\begin{scope}[xshift=-2cm, yshift=-0.5cm]
    \draw[fill] (3, 0) circle (1pt) node(a02){};
\end{scope}
\draw (4, -0.5) node[anchor=north east] (j5){\quad and \quad};
\draw[thick] (4.5,-0.5) rectangle (6.5,-1.5);
\begin{scope}[xshift=2cm, yshift=-1cm]
    \draw (3, 0) node(a02){$*$};
    \draw[fill] (4, 0) circle (1pt) node (a03){};
    \draw[mono] (a02)--(a03);
\end{scope}
\draw (7.5, -0.7) node[anchor=north east] (j5){{$\mapsto$}};
\draw[thick] (7.5,-0.5) rectangle (9.5,-1.5);
\begin{scope}[xshift=6cm, yshift=-1cm]
    \draw[fill] (3, 0) circle (1pt) node(a02){}; %
\end{scope}
\end{tikzpicture}
\end{center}
\end{enumerate}
\end{rmk}

The definition is general enough to recover a variety of examples of interest, typically through a nerve construction. Let us discuss the construction from \cite[\textsection2]{BOORS4} in the case of $\cE$ being the category of Abelian groups, or more generally an \emph{Abelian category} (in the sense of \cite{MacLane}), or more generally an \emph{exact category} (in the sense of \cite{Barr}).

\begin{const}
    Given an exact category $\cE$, the preaugmented bisimplicial space $N^{\mathrm{ex}}\cE$ is defined as follows. The space $N_{-1}^{\mathrm{ex}}\cE$ is set to be the nerve of the groupoid of zero objects in $\cE$, and for $a,b\geq0$ we set $N^{\mathrm{ex}}_{a,b}\cE$ to be the nerve of the groupoid of functors $[a]\times[b]\to\cE$ such that
    \begin{itemize}[leftmargin=*]
    \item all the restrictions $[0]\times[1]\to\cE$ are admissible monomorphisms in $\cE$,
    \item all the restrictions $[1]\times[0]\to\cE$ are admissible epimorphisms in $\cE$,
    \item all the restrictions $[1]\times[1]\to\cE$ are bicartesian squares in $\cE$. 
    \end{itemize}
    The structure maps are the obvious ones.
\end{const}

\begin{prop}[{\cite[\textsection2]{BOORS4}}]
    Given an exact category $\cE$, the preaugmented bisimplicial space $N^{\mathrm{ex}}\cE$ is a stable augmented double Segal space.
\end{prop}

\begin{proof}[Proof idea] 
The fact that the Segal maps
 \[(d_2^v,d_0^v)\colon N^{\mathrm{ex}}_{2,1}\cE\to N^{\mathrm{ex}}_{1,1}\cE\times^h_{N^{\mathrm{ex}}_{0,1}\cE}N^{\mathrm{ex}}_{1,1}\cE\quad\text{and}\quad(d_2^h,d_0^h)\colon N^{\mathrm{ex}}_{1,2}\cE\to N^{\mathrm{ex}}_{1,1}\cE\times^h_{N^{\mathrm{ex}}_{1,0}\cE}N^{\mathrm{ex}}_{1,1}\cE\]
 define equivalences of spaces is a consequence of the property of pullback and pushout
 cancellation in $\cE$.
 
The fact that the stability maps
 \[(d^v_1,d^h_1)\colon N^{\mathrm{ex}}_{1,1}\cE\to N^{\mathrm{ex}}_{0,1}\cE\times^h_{N^{\mathrm{ex}}_{0,0}\cE}N^{\mathrm{ex}}
_{1,0}\cE\ \text{and}\ (d^h_0,d^v_0)\colon N^{\mathrm{ex}}_{1,1}\cE\to N^{\mathrm{ex}}_{1,0}\cE\times^h_{N^{\mathrm{ex}}_{0,0}\cE}N_{0,1}^{\mathrm{ex}}
\cE\]
define equivalences of spaces holds because
every pullback (resp.~pushout) square in $\cE$ is determined up to isomorphism by the span (resp.~cospan) contained in its boundary.

The fact that the augmentation maps
 \[d^v_1\colon N^{\mathrm{ex}}_{1,0}\cE\times^h_{N^{\mathrm{ex}}_{0,0}\cE}N^{\mathrm{ex}}_{-1}\cE\xrightarrow{\simeq}N^{\mathrm{ex}}_{0,0}\cE\quad\text{and}\quad
 d^h_0\colon N^{\mathrm{ex}}_{0,1}\cE\times^h_{N^{\mathrm{ex}}_{0,0}\cE}N^{\mathrm{ex}}_{-1}\cE\xrightarrow{\simeq}N^{\mathrm{ex}}_{0,0}\cE\]
  define equivalences of spaces is a consequence of the fact that any zero object in $\cE$ is initial in the category of admissible monomorphisms and  terminal in the category of admissible epimorphisms of $\cE$.
\end{proof}

One can mimic the same nerve construction and corresponding argument in various other contexts:

\begin{rmk}
    Stable augmented double Segal spaces arise via an appropriate nerve constructions in many situations:
    \begin{itemize}[leftmargin=*]
        \item Given a (proto-)exact ($\infty$-)category (in the sense of \cite{BarwickKtheory,DKbook}), a nerve construction in the form of a stable augmented double Segal space is provided in \cite[\textsection2,~\textsection4]{BOORS4}.
        \item Given a stable $\infty$-category (in the sense of \cite{LurieStable}), a nerve construction resulting in a stable augmented double Segal space is provided in \cite[\textsection3]{BOORS4}.
        \item Given an appropriate flavor of CGW-categories (in the sense of \cite{CZdevissage}), a nerve construction resulting in a stable augmented double Segal space should exist, and is currently the subject of further investigation by some of the workshop participants (cf.~\cite{Oberwolfach2024}).
        \item Warning: Given a Waldhausen category (in the sense of \cite{waldhausen}), the canonical choice for a nerve construction results in a preaugmented bisimplicial space which is essentially never a stable augmented double Segal space. Precisely, it satisfies the double Segality condition, but only satisfies one of the stability conditions and one of the augmentation conditions.
        \item Given a $2$-Segal space (as recalled in \cref{Section2SegalSpaces}), we will define a path construction which will result in a stable augmented double Segal space in \cref{SectionPathConstruction}.
    \end{itemize}
\end{rmk}

Again, in order to state the main result we shall consider a categorical structure for the class of stable augmented double Segal spaces. As a first attempt, one can consider an ordinary category:

\begin{rmk}
There is a category $\mathrm{saDblSegSp}$ in which:
\begin{itemize}[leftmargin=*]
    \item the set of objects is given by the stable augmented double Segal spaces;
    \item the morphisms are the maps of simplicial spaces between them; and
    \item composition and identity operators are inherited from those in the category of preaugmented bisimplicial spaces.
\end{itemize} 
\end{rmk}

We can add homotopy theory by considering equivalences of stable augmented double Segal spaces:

\begin{defn}
Let $\cD$ and $\cE$ be stable augmented double Segal spaces. A map of preaugmented bisimplicial spaces $\varphi\colon \cD\to\cE$ is an
\emph{equivalence} of stable augmented double Segal spaces if for all $a,b\geq0$ it defines an equivalence of spaces $\varphi_{a,b}\colon \cD_{a,b}\xrightarrow{\simeq} \cE_{a,b}$, as well as an equivalence of spaces $\varphi_{-1}\colon \cD_{-1}\xrightarrow{\simeq} \cE_{-1}$.
\end{defn}

With this notion of equivalence between stable augmented double Segal spaces, one can consider the set of equivalence classes of stable augmented double Segal spaces:

\begin{rmk}
Consider the relation on the class of stable augmented double Segal spaces given by declaring that $\cD$ is related to $\cE$ if there exists an equivalence $\varphi\colon \cD\xrightarrow{\simeq}\cE$. It follows from the definitions that this relation is reflexive and transitive, though without imposing further technical conditions it is not a priori symmetric.
Nevertheless, one can consider the set $\mathrm{saDblSegSp}/_\simeq$ of stable augmented double Segal spaces up to the equivalence relation $\simeq$ it generates.
\end{rmk}

By simplicially localizing the category of stable augmented double Segal spaces at the class of equivalences of stable augmented double Segal spaces, we obtain a simplicial category in which all equivalences of stable augmented double Segal spaces are universally inverted in an appropriate weak sense:

\begin{const}
\label{saDblSegSp}
The simplicial localization of the category of stable augmented double Segal spaces and equivalences stable augmented double Segal spaces
is a simplicial category $\bm{\mathrm{saDblSegSp}}$ in which:
    \begin{itemize}[leftmargin=*]
        \item The set of objects is given by the collection of the stable augmented double Segal spaces;
        \item Given stable augmented double Segal spaces $\cD$ and $\cE$, the mapping space of maps from a $\cD$ to $\cE$ is given by the space of \emph{hammocks} from $\cD$ to $\cE$.
        \item Composition and identity operators are given by concatenating hammocks, and taking the trivial hammock, respectively.
    \end{itemize}
\end{const}

\section{The path construction}

\label{SectionPathConstruction}

In this section we introduce the \emph{path construction} $\cP$ of a $2$-Segal space from \cite{BOORS3}, and show it is a  stable augmented double Segal space.

\begin{const}
\label{PathSADSS}
Given a simplicial space $X$, we define the \emph{path construction} of $X$, which is a preaugmented bisimplicial space $\cP X$. It consists of the spaces
\[\cP_{-1} X\coloneqq X_0\quad\text{ and }\quad \cP_{a,b} X\coloneqq X_{a+1+b}\quad\text{ for }a,b\geq0,\]
together with the following structure maps:
\begin{itemize}[leftmargin=*]
        \item the vertical $i$-th face map $d^v_i\colon\cP_{n,\ell}X\to\cP_{n-1,\ell}X$ is given by $d_i\colon X_{n+1+\ell}\to X_{n+\ell}$ for $\ell\geq0$ and $n\geq0$,
         \item the horizontal $i$-th face map $d^h_i\colon\cP_{\ell,n}X\to\cP_{\ell,n-1}X$ is given by $d_{i+1+\ell}\colon X_{\ell+1+n}\to X_{\ell+n}$ for $\ell\geq0$ and $n\geq0$,
\item the vertical $i$-th degeneracy map $s^v_i\colon\cP_{n,\ell}X\to\cP_{n+1,\ell}X$ is given by $s_i\colon X_{n+1+\ell}\to X_{n+2+\ell}$ for $\ell\geq0$ and $n>0$,
         \item the horizontal $i$-th degeneracy map $s^h_i\colon\cP_{\ell,n}X\to\cP_{\ell,n+1}X$ is given by $s_{i+1+\ell}\colon X_{\ell+1+n}\to X_{\ell+2+n}$ for $\ell\geq0$ and $n>0$,
         \item the augmentation map $\varepsilon\colon\cP_{-1}X\to\cP_{0,0}X$ is given by $s_0\colon X_0\to X_1$.
\end{itemize}

To provide more intuition, let us spell out some of the basic data of $\cP X$:

\begin{rmk}
Let $X$ be a simplicial space.
    \begin{itemize}[leftmargin=*]
        \item The space of objects $\cP_{0,0} X$ is given by the space of $1$-simplices of $X$.
        \item The augmentation space $\cP_{-1} X$ is given by the $0$-simplices of $X$, regarded as degenerate $1$-simplices.
        \item The spaces of vertical and horizontal morphisms $\cP_{1,0} X$ and $\cP_{0,1} X$ are both given by the space of $2$-simplices of $X$, but the structure maps act differently. Indeed, a $2$-simplex $\sigma$ defines at the same time a vertical morphism and a horizontal morphism which can be depicted as follows:
        \begin{center}
    \begin{tikzpicture}[scale=1.5]
        \draw[thick] (3.1,0.2) rectangle (3.9,-1.2);
        \begin{scope}[xshift=3.5cm, yshift=-0.5cm]
            \draw (0, -0.5)  node(a12){$d_0\sigma$};
            \draw (0, 0.5)  node(a02){$d_1\sigma$};
            \draw[epi] (a02)--(a12);
        \end{scope}
        \draw (6.2, -0.25) node[anchor=north east] (j5){\quad and \quad};
        \draw[thick] (7.5,-0.2) rectangle (9.5,-0.8);
        \begin{scope}[xshift=5cm, yshift=-0.5cm]
            \draw (3, 0)  node(a02){$d_2\sigma$};
            \draw (4, 0) node (a03){$d_1\sigma$};
            \draw[mono] (a02)--(a03);
        \end{scope}
    \end{tikzpicture}
\end{center}
\item The space of squares $\cP_{1,1} X$, and the spaces of pairs of composable and horizontal morphisms $\cP_{0,2} X$ and $\cP_{2,0} X$ are all given by the space of $3$-simplices of $X$, but the structure maps act differently. Indeed, a $3$-simplex $\tau$ defines at the same time a pair of composable vertical morphisms, a square, and a pair of composable horizontal morphisms which can be depicted as follows:
 
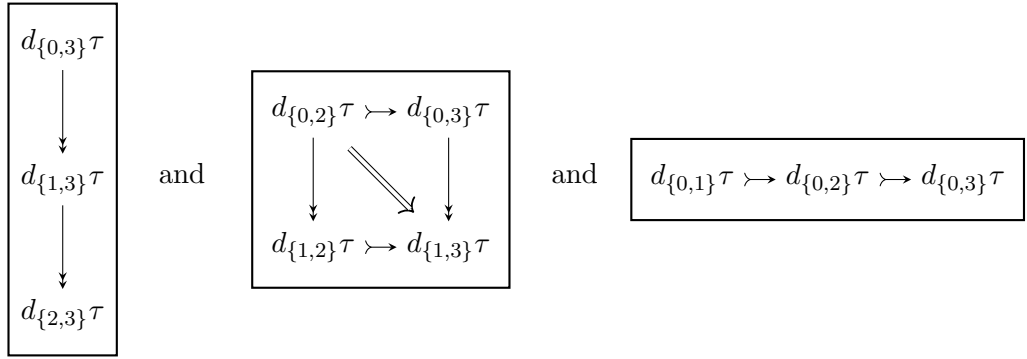
\begin{center}
    \begin{tikzpicture}[scale=1.8]
     \draw[thick] (1.3,0.3) rectangle (2.1,-2.3);
    \begin{scope}[xshift=-1.3cm, yshift=0cm]
        \draw (3, 0)  node(a02){$d_{\{0,3\}}\tau$};
        \draw (3, -1)  node(a12){$d_{\{1,3\}}\tau$};
        \draw (3,-2)  node (a22){$d_{\{2,3\}}\tau$};
        \draw[epi] (a12)--(a22);
        \draw[epi] (a02)--(a12);
    \end{scope}
    \draw (2.9, -0.8) node[anchor=north east] (j5){\quad and \quad};
        \draw[thick] (3.1,-0.2) rectangle (5,-1.8);
        \begin{scope}[xshift=0.55cm, yshift=-0.5cm]
            \draw (3, -1)  node(a12){$d_{\{1,2\}}\tau$};
            \draw (3, 0)  node(a02){$d_{\{0,2\}}\tau$};
            \draw (4, 0)  node (a03){$d_{\{0,3\}}\tau$};
            \draw (4, -1) node (a13){$d_{\{1,3\}}\tau$};
            \draw[mono] (a02)--(a03);
            \draw[mono] (a12)--(a13);
            \draw[epi] (a02)--(a12);
            \draw[epi] (a03)--(a13);
            \begin{scope}[yshift=-0.3cm]
                \draw[twoarrowlonger] (3.2,0.1)--(3.8,-0.5);
            \end{scope}
        \end{scope}
        \draw (5.8, -0.8) node[anchor=north east] (j5){\quad and \quad};   
        \draw[thick] (5.9,-0.7) rectangle (8.8,-1.3);
     \begin{scope}[xshift=5.35cm, yshift=-1cm]
        \draw (1,0) node(a00){$d_{\{0,1\}}\tau$};
        \draw (2,0)  node(a01){$d_{\{0,2\}}\tau$};
        \draw (3, 0)  node(a02){$d_{\{0,3\}}\tau$};
        \draw[mono] (a00)--(a01);
        \draw[mono] (a01)--(a02);
    \end{scope}
    \end{tikzpicture}
\end{center}
    \end{itemize}
\end{rmk}

\end{const}
\begin{prop}
    If $X$ is a $2$-Segal space, the preaugmented bisimplicial space $\cP X$ is a stable augmented double Segal space.
\end{prop}

\begin{proof}
Since $X$ is a $2$-Segal space, by \cref{2SegalSpace,unital} for all $n\geq2$ there are homotopy pullback squares of spaces
\[
\begin{tikzcd}
X_{3}\arrow[rd, phantom, "\lrcorner", very near start]\arrow[r,"d_0"]\arrow[d,"d_{2}" swap]&X_2\arrow[d,"d_1"]\\
X_{1}\arrow[r,"d_0" swap]&X_1
\end{tikzcd}
\quad\text{ and }\quad
\begin{tikzcd}
X_3\arrow[rd, phantom, "\lrcorner", very near start]\arrow[r,"d_3" ]\arrow[d,"d_1" swap]&X_2\arrow[d,"d_1"]\\
X_2\arrow[r,"d_2" swap]&X_1
\end{tikzcd}
\quad\text{ and }\quad
\begin{tikzcd}
X_1\arrow[rd, phantom, "\lrcorner", very near start]\arrow[r,"s_0" ]\arrow[d,"d_1" swap]&X_2\arrow[d,"d_2"]\\
X_0\arrow[r,"s_0" swap]&X_1
\end{tikzcd}
\]
By definition of $\cP X$, these can be rewritten as homotopy pullback squares of spaces
\[
\begin{tikzcd}
\cP_{2,0} X\arrow[rd, phantom, "\lrcorner", very near start]\arrow[r,"d^v_0" ]\arrow[d,"d^v_2"swap]&\cP_{1,0} X\arrow[d,"d^v_1" ]\\
\cP_{1,0} X\arrow[r,"d^v_0" swap]&\cP_{0,0} X
\end{tikzcd}
\text{ and }
\begin{tikzcd}
\cP_{1,1} X\arrow[rd, phantom, "\lrcorner", very near start]\arrow[r,"d^h_1"]\arrow[d,"d^v_1" swap]&\cP_{1,0} X\arrow[d,"d^v_1"]\\
\cP_{0,1} X\arrow[r,"d^h_1" swap]&\cP_{0,0} X
\end{tikzcd}
\text{ and }
\begin{tikzcd}
\cP_{0,0} X\arrow[rd, phantom, "\lrcorner", very near start]\arrow[r,""]\arrow[d,"" swap]&\cP_{0,1} X\arrow[d,"d^h_1"]\\
\cP_{-1} X\arrow[r,"\varepsilon" swap]&\cP_{0,0} X
\end{tikzcd}
\]
These imply, respectively, the first smallest instances of the conditions of double Segality, stability and (after thinking about it a little bit) augmentation
from \cref{sadss} for $\cP X$. The dual conditions and the general property of double Segality can be proven with a similar reasoning.
\end{proof}

We can unpack the path construction $\cP\Delta[n]$ of $\Delta[n]$ for $n\leq4$, and let the reader imagine that of $\Delta[n]$ for general $n$.

\begin{ex}
Following the conventions established in \cref{Interpretation}, the path constructions
\[\cP\Delta[0]\quad,\quad\cP\Delta[1]\quad,\quad\cP\Delta[2]\quad,\quad\cP\Delta[3]\quad,\quad\cP\Delta[4]\]
are generated by the following data, respectively:
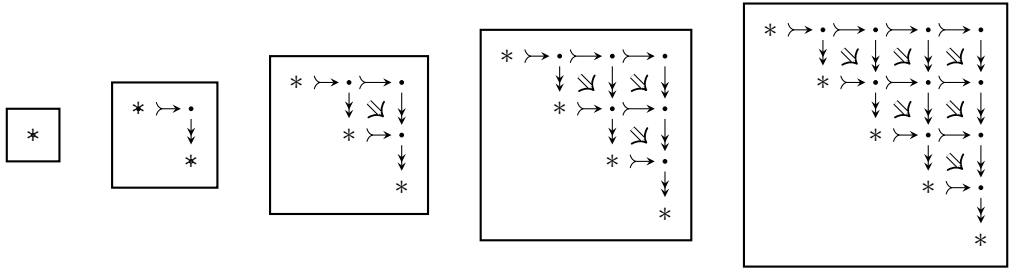
\begin{center}
\begin{tikzpicture}[scale=0.7]
 \draw[thick] (-8,-0.5) rectangle (-7,-1.5);
        \begin{scope}[xshift=-10cm, yshift=-1.5cm]
            \draw (2.5, 0.5) circle (1pt) node(a12){$*$};
        \end{scope}
    \draw[thick] (-6.0,0.0) rectangle (-4.0,-2.0);
    \begin{scope}[xshift=-7.5cm, yshift=-0.5cm]
        \draw[fill] (2,0) circle (1pt) node(a01){$*$};
        \draw (3, -1) circle (1pt) node(a12){$*$};
        \draw[fill] (3, 0) circle (1pt) node(a02){};
        \draw[mono] (a01)--(a02);
        \draw[epi] (a02)--(a12);
    \end{scope}
    \draw[thick] (-3,0.5) rectangle (0,-2.5);
    \begin{scope}[xshift=-3.5cm, yshift=0cm]
        \draw (1,0) node(a00){$*$};
        \draw[fill] (2,0) circle (1pt) node(a01){};
        \draw (2,-1) node(a11) {$*$};
        \draw[fill] (3, -1) circle (1pt) node(a12){};
        \draw[fill] (3, 0) circle (1pt) node(a02){};
        \draw (3,-2) node (a22){$*$};
        \draw[mono] (a00)--(a01);
        \draw[mono] (a11)--(a12);
        \draw[mono] (a01)--(a02);
        \draw[epi] (a01)--(a11);
        \draw[epi] (a12)--(a22);
        \draw[epi] (a02)--(a12);
        \begin{scope}[yshift=-0.3cm]
            \draw[twoarrowlonger] (2.2,0.1)--(2.8,-0.5);
        \end{scope}
    \end{scope}
    \draw[thick] (1,1) rectangle (5,-3);
    \begin{scope}[xshift=0.5cm, yshift=0.5cm]
        \draw (1,0) node(a00){$*$};
        \draw[fill] (2,0) circle (1pt) node(a01){};
        \draw (2,-1) node(a11) {$*$};
        \draw[fill] (3, -1) circle (1pt) node(a12){};
        \draw[fill] (3, 0) circle (1pt) node(a02){};
        \draw (3,-2) node (a22){$*$};
        \draw[fill] (4, 0) circle (1pt) node (a03){};
        \draw[fill] (4, -1) circle (1pt) node (a13){};
        \draw[fill](4, -2) circle (1pt) node (a23){};
        \draw (4, -3) node (a33){$*$};
        \draw[mono] (a00)--(a01);
        \draw[mono] (a11)--(a12);
        \draw[mono] (a01)--(a02);
        \draw[mono] (a02)--(a03);
        \draw[mono] (a12)--(a13);
        \draw[mono] (a22)--(a23);
        \draw[epi] (a01)--(a11);
        \draw[epi] (a12)--(a22);
        \draw[epi] (a02)--(a12);
        \draw[epi] (a03)--(a13);
        \draw[epi] (a13)--(a23);
        \draw[epi] (a23)--(a33);
        \begin{scope}[yshift=-0.3cm]
            \draw[twoarrowlonger] (2.2,0.1)--(2.8,-0.5);
            \draw[twoarrowlonger] (3.2,0.1)--(3.8,-0.5);
            \draw[twoarrowlonger] (3.2,-0.9)--(3.8,-1.5);
        \end{scope}
    \end{scope}
    \draw[thick] (6,1.5) rectangle (11,-3.5);
    \begin{scope}[xshift=5.5cm,yshift=1cm]
        \draw (1,0) node(a00){$*$};
        \draw[fill] (2,0) circle (1pt) node(a01){};
        \draw (2,-1) node(a11) {$*$};
        \draw[fill] (3, -1) circle (1pt) node(a12){};
        \draw[fill] (3, 0) circle (1pt) node(a02){};
        \draw (3,-2) node (a22){$*$};
        \draw[fill] (4, 0) circle (1pt) node (a03){};
        \draw[fill] (4, -1) circle (1pt) node (a13){};
        \draw[fill](4, -2) circle (1pt) node (a23){};
        \draw (4, -3) node (a33){$*$};
        \draw[fill] (5, 0) circle (1pt) node (a04){};
        \draw[fill] (5, -1)circle (1pt) node (a14){};
        \draw[fill] (5, -2)circle (1pt) node (a24){};
        \draw[fill] (5, -3)circle(1pt) node (a34){};
        \draw (5, -4) node (a44){$*$};
        \draw[mono] (a00)--(a01);
        \draw[mono] (a11)--(a12);
        \draw[mono] (a01)--(a02);
        \draw[mono] (a02)--(a03);
        \draw[mono] (a12)--(a13);
        \draw[mono] (a22)--(a23);
        \draw[mono] (a03)--(a04);
        \draw[mono] (a13)--(a14);
        \draw[mono] (a23)--(a24);
        \draw[mono] (a33)--(a34);
        \draw[epi] (a01)--(a11);
        \draw[epi] (a12)--(a22);
        \draw[epi] (a02)--(a12);
        \draw[epi] (a03)--(a13);
        \draw[epi] (a13)--(a23);
        \draw[epi] (a23)--(a33);
        \draw[epi] (a04)--(a14);
        \draw[epi] (a14)--(a24);
        \draw[epi] (a24)--(a34);
        \draw[epi] (a34)--(a44);
        \begin{scope}[yshift=-0.3cm]
            \draw[twoarrowlonger] (2.2,0.1)--(2.8,-0.5);
            \draw[twoarrowlonger] (3.2,0.1)--(3.8,-0.5);
            \draw[twoarrowlonger] (3.2,-0.9)--(3.8,-1.5);
            \draw[twoarrowlonger] (4.2,0.1)--(4.8,-0.5);
            \draw[twoarrowlonger] (4.2,-1.9)--(4.8,-2.5);
            \draw[twoarrowlonger] (4.2,-0.9)--(4.8,-1.5);
        \end{scope}
    \end{scope}
   

\end{tikzpicture}
\end{center}
\end{ex}

\begin{rmk}
The path construction induces a functor
\[
\cP\colon\mathrm{2SegSp}\to\mathrm{saDblSegSp}.
\]
\end{rmk}

This functor is compatible with the homotopy theory of the objects involved:

\begin{prop}
\label{Phomotopical}
If $f\colon X\xrightarrow{\simeq} Y$ is an equivalence of $2$-Segal spaces, then
\[\cP f\colon\cP X\xrightarrow{\simeq}\cP Y\]
is an equivalence of stable augmented double Segal spaces.
\end{prop}

\begin{proof}
Since $f\colon X\xrightarrow{\simeq}Y$ is an equivalence of stable augmented double Segal spaces, for all $a,b\geq0$ there are equivalences of spaces
\[f_{a+1+b}\colon X_{a+1+b}\xrightarrow{\simeq} Y_{a+1+b}\quad\text{ and }\quad
f_{0}\colon X_{0}\xrightarrow{\simeq} Y_{0}.\]
These by construction are equivalences of spaces
\[\cP_{a,b}f\colon\cP_{a,b} X\xrightarrow{\simeq}\cP_{a,b} Y\quad\text{ and }\quad
\cP_{-1} f\colon\cP X_{-1}\xrightarrow{\simeq}\cP_{-1} Y.\]
So $\cP f\colon\cP X\xrightarrow{\simeq}\cP Y$ is an equivalence of stable augmented double Segal spaces, as desired.
\end{proof}

Thanks to \cref{Phomotopical}, we obtain that the path construction descends to equivalence classes:

\begin{rmk}
 The path construction defines a function
        \[
        \cP\colon
  \mathrm{2SegSp}/_\simeq\ \to\ \mathrm{saDblSegSp}/_\simeq.
  \]
\end{rmk}

Moreover, it follows from \cite[\textsection1.5]{BarwickKan} that the path construction also induces a functor at the level of simplicial localizations:

\begin{const}
The path construction induces a simplicial functor
\[
{\bm{\cP}}\colon\bm{\mathrm{2SegSp}}\to\bm{\mathrm{saDblSegSp}}.
\]
\end{const}

\section{The $S_\bullet$-construction}

In this section we introduce the \emph{$S_\bullet$-construction} of a stable augmented double Segal space from \cite{BOORS3}, and show it is a $2$-Segal space
under mild technical assumptions.

In order to define the $S_\bullet$-construction, we should introduce some further structure for the category of preaugmented bisimplicial spaces.

\begin{notn}
Given a preaugmented bisimplicial space $\cA$ and a simplicial space $S$, the \emph{tensor} $\cA\boxtimes S$ of $\cA$ with $S$ is the preaugmented
bisimplicial space which consists of the spaces
\[
(\cA\boxtimes S)_{-1}\coloneqq\cA_{-1}\times S
\quad\text{ and }
(\cA\boxtimes S)_{a,b}\coloneqq\cA_{a,b}\times S\quad\text{for $a,b\geq0$}.\]
The structures maps are the obvious ones.
\end{notn}

\begin{notn}
\label{MappingSpaces}
Given preaugmented bisimplicial spaces $\cA$ and $\cB$, the \emph{mapping space} $\Map(\cA,\cB)$ from $\cA$ to $\cB$ is the space given, for $\ell\geq0$, by the set
\[
\Map_\ell(\cA,\cB)\coloneqq\Hom_{\mathrm{saDblSegSp}}(\cA\boxtimes\Delta[\ell],\cB).\]
The structures maps are the obvious ones.
\end{notn}

We can now define the $S_\bullet$-construction of a preaugmented bisimplicial space. We regard the standard simplex $\Delta[n]$ as a discrete simplicial space. 

\begin{const}
    Given a preaugmented bisimplicial space $\cD$, we define the \emph{$S_\bullet$-construction} of $X$, which is a simplicial space $S_{\bullet}\cD$. It consists, for $n\geq0$,
    of the space
    \[  S_n\cD:=\Map(\cP\Delta[n],\cD).
    \]
    The structure maps are the obvious ones, induced by the cosimplicial structure of $n\mapsto\Delta[n]\mapsto\cP\Delta[n]$.
\end{const}

We'd like to say that the $S_\bullet$-construction of a stable augmented double Segal space is a $2$-Segal space, but this fails in such generality.
The problem can be resolved by incorporating a technical condition, which we acknowledge for the sake of accuracy; however, we encourage readers who are not acquainted with it to disregard this detail in order to focus on the primary narrative.

We refer the reader to \cite[\textsection6.11]{Hirschhorn} for the notion of an \emph{injectively fibrant} preaugmented bisimplicial space (i.e., a fibrant object in the injective model structure of $\Sigma$-spaces.
What one needs to know is just that for every preaugmented simplicial space $\cD$ there exists an injectively fibrant preaugmented bisimplicial space $\widetilde{\cD}$ and a natural map $\cD\to\widetilde{\cD}$ which induces an equivalence of spaces $\cD_{-1}\to\widetilde{\cD}_{-1}$ and $\cD_{a,b}\to\widetilde{\cD}_{a,b}$ for $a,b\geq0$. Moreover, if $\cD$ is a stable augmented bisimplicial space, so is $\widetilde\cD$.

\begin{prop}
\label{Sdot2Segal}
    If $\cD$ is an injectively fibrant stable augmented double Segal space, the simplicial space $S_{\bullet}\cD$ is a $2$-Segal space.
\end{prop}

Proving the proposition amounts to analyzing the $2$-Segal maps of $S_\bullet\cD$:

\begin{lem}
\label{Lemma2Segal}
Let $\cD$ be an injectively fibrant
stable augmented double Segal space.
For all $n\geq0$ the second canonical map from \cref{2SegalSpace} is an equivalence of spaces
\[S_{n+1}\cD\xrightarrow{\simeq}S_n\cD\times^h_{S_1\cD}S_2\cD.\]
For instance, when $n=2$ it can be depicted as\footnote{In this specific instance, the graphical calculus fails to be clear on whether composites shall be included. However, in presence of the property double Segality for $\cD$, this ambiguity does not undermine the argument or the underlying intuition.}
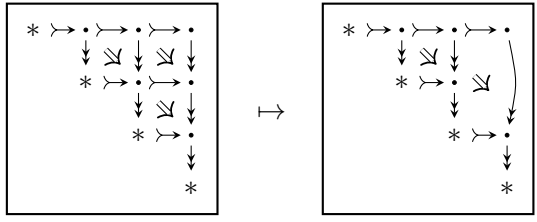
\begin{center}
    \begin{tikzpicture}[scale=0.7]
    \draw[thick] (6.5,0.5) rectangle (10.5,-3.5);
    \draw[thick] (0.5,0.5) rectangle (4.5,-3.5);
\begin{scope}[xshift=0cm, yshift=0cm]
     \draw (1,0) node(a00){$*$};
\draw[fill] (2,0) circle (1pt) node(a01){};
\draw (2,-1) node(a11) {$*$};
\draw[fill] (3, -1) circle (1pt) node(a12){};
\draw[fill] (3, 0) circle (1pt) node(a02){};
\draw (3,-2) node (a22){$*$};
\draw[fill] (4, 0) circle (1pt) node (a03){};
\draw[fill] (4, -1) circle (1pt) node (a13){};
\draw[fill](4, -2) circle (1pt) node (a23){};
\draw (4, -3) node (a33){$*$};
\draw[mono] (a00)--(a01);
\draw[mono] (a11)--(a12);
\draw[mono] (a01)--(a02);
\draw[mono] (a02)--(a03);
\draw[mono] (a12)--(a13);
\draw[mono] (a22)--(a23);
\draw[epi] (a01)--(a11);
\draw[epi] (a12)--(a22);
\draw[epi] (a02)--(a12);
\draw[epi] (a03)--(a13);
\draw[epi] (a13)--(a23);
\draw[epi] (a23)--(a33);
\begin{scope}[yshift=-0.3cm]
  \draw[twoarrowlonger] (2.2,0.1)--(2.8,-0.5);
  \draw[twoarrowlonger] (3.2,0.1)--(3.8,-0.5);
  \draw[twoarrowlonger] (3.2,-0.9)--(3.8,-1.5);
\end{scope}
\end{scope}

\draw (6, -1.3) node[anchor=north east] (j5){$\mapsto$};
\begin{scope}[xshift=6cm, yshift=0cm]
     \draw (1,0) node(a00){$*$};
\draw[fill] (2,0) circle (1pt) node(a01){};
\draw (2,-1) node(a11) {$*$};
\draw[fill] (3, -1) circle (1pt) node(a12){};
\draw[fill] (3, 0) circle (1pt) node(a02){};
\draw (3,-2) node (a22){$*$};
\draw[fill] (4, 0) circle (1pt) node (a03){};
\draw[fill](4, -2) circle (1pt) node (a23){};
\draw (4, -3) node (a33){$*$};
\draw[mono] (a00)--(a01);
\draw[mono] (a11)--(a12);
\draw[mono] (a01)--(a02);
\draw[mono] (a02)--(a03);
\draw[mono] (a22)--(a23);
\draw[epi] (a01)--(a11);
\draw[epi] (a12)--(a22);
\draw[epi] (a02)--(a12);
 \draw[epi] (a03)..controls (4.2,-1)..(a23);
\draw[epi] (a23)--(a33);
\begin{scope}[yshift=-0.3cm]
  \draw[twoarrowlonger] (2.2,0.1)--(2.8,-0.5);
  \draw[twoarrowlonger] (3.2,-0.4)--(3.8,-1.0);
\end{scope}
\end{scope}
    \end{tikzpicture}
\end{center}
\end{lem}

We illustrate the idea in the case $n=2$. An analog (though much more intricate) argument could be provided for the case of arbitrary $n$ using \cite[Lemma 5.12]{BOORS3} and an appropriate generalization of- \cite[Proposition~2.4.8]{DKbook}.

\begin{proof}[Idea of the proof of \cref{Lemma2Segal}]
There are equivalences of spaces of the form:
\[S_{n+1}\cD\xrightarrow{\simeq}S_n\cD\times^h_{S_1\cD}S_2\cD,\]
which can be constructed as follows:
\begin{itemize}[leftmargin=*]
    \item First, there is an equivalence of spaces
    \[S_{n+1}\cD\xrightarrow{\simeq}S_n\cD\times^h_{S_1\cD}S_2\cD\times^h_{\cD_{1,1}\times_{\cD_{0,0}}\cD_{0,2}}\cD_{1,2}=:F,\]
    which essentially forgets the bottom square and one of the vertical morphisms, and can be depicted as
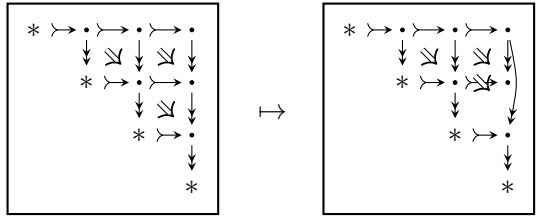
\begin{center}
    \begin{tikzpicture}[scale=0.7]
    \draw[thick] (6.5,0.5) rectangle (10.5,-3.5);
    \draw[thick] (0.5,0.5) rectangle (4.5,-3.5);
\begin{scope}[xshift=0cm, yshift=0cm]
     \draw (1,0) node(a00){$*$};
\draw[fill] (2,0) circle (1pt) node(a01){};
\draw (2,-1) node(a11) {$*$};
\draw[fill] (3, -1) circle (1pt) node(a12){};
\draw[fill] (3, 0) circle (1pt) node(a02){};
\draw (3,-2) node (a22){$*$};
\draw[fill] (4, 0) circle (1pt) node (a03){};
\draw[fill] (4, -1) circle (1pt) node (a13){};
\draw[fill](4, -2) circle (1pt) node (a23){};
\draw (4, -3) node (a33){$*$};
\draw[mono] (a00)--(a01);
\draw[mono] (a11)--(a12);
\draw[mono] (a01)--(a02);
\draw[mono] (a02)--(a03);
\draw[mono] (a12)--(a13);
\draw[mono] (a22)--(a23);
\draw[epi] (a01)--(a11);
\draw[epi] (a12)--(a22);
\draw[epi] (a02)--(a12);
\draw[epi] (a03)--(a13);
\draw[epi] (a13)--(a23);
\draw[epi] (a23)--(a33);
\begin{scope}[yshift=-0.3cm]
  \draw[twoarrowlonger] (2.2,0.1)--(2.8,-0.5);
  \draw[twoarrowlonger] (3.2,0.1)--(3.8,-0.5);
  \draw[twoarrowlonger] (3.2,-0.9)--(3.8,-1.5);
\end{scope}
\end{scope}

\draw (6, -1.3) node[anchor=north east] (j5){$\mapsto$};
\begin{scope}[xshift=6cm, yshift=0cm]
     \draw (1,0) node(a00){$*$};
\draw[fill] (2,0) circle (1pt) node(a01){};
\draw (2,-1) node(a11) {$*$};
\draw[fill] (3, -1) circle (1pt) node(a12){};
\draw[fill] (3, 0) circle (1pt) node(a02){};
\draw (3,-2) node (a22){$*$};
\draw[fill] (4, 0) circle (1pt) node (a03){};
\draw[fill] (4, -1) circle (1pt) node (a13){};
\draw[fill](4, -2) circle (1pt) node (a23){};
\draw (4, -3) node (a33){$*$};
\draw[mono] (a00)--(a01);
\draw[mono] (a11)--(a12);
\draw[mono] (a01)--(a02);
\draw[mono] (a02)--(a03);
\draw[mono] (a12)--(a13);
\draw[mono] (a22)--(a23);
\draw[epi] (a01)--(a11);
\draw[epi] (a12)--(a22);
\draw[epi] (a02)--(a12);
\draw[epi] (a03)--(a13);
 \draw[epi] (a03)..controls (4.2,-1)..(a23);
\draw[epi] (a23)--(a33);
\begin{scope}[yshift=-0.3cm]
  \draw[twoarrowlonger] (2.2,0.1)--(2.8,-0.5);
  \draw[twoarrowlonger] (3.2,0.1)--(3.8,-0.5);
    \draw[twoarrowlonger] (3.2,-0.4)--(3.8,-1.0);
\end{scope}
\end{scope}
    \end{tikzpicture}
\end{center}
and can be understood to be an equivalence of spaces using the properties of double Segality and stability for $\cD$.
    \item Next, there is an equivalence of spaces
    \[F\xrightarrow{}S_n\cD\times^h_{S_1\cD}S_2\cD,\]
    which essentially forget the top right square and the cospan contained in its boundary, and can be depicted as
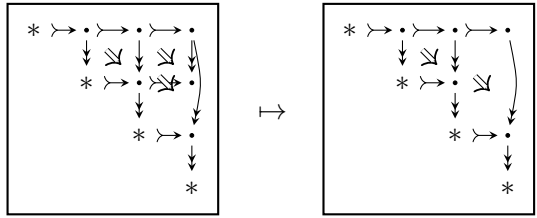
\begin{center}
    \begin{tikzpicture}[scale=0.7]
    \draw[thick] (6.5,0.5) rectangle (10.5,-3.5);
    \draw[thick] (0.5,0.5) rectangle (4.5,-3.5);
\begin{scope}[xshift=0cm, yshift=0cm]
     \draw (1,0) node(a00){$*$};
\draw[fill] (2,0) circle (1pt) node(a01){};
\draw (2,-1) node(a11) {$*$};
\draw[fill] (3, -1) circle (1pt) node(a12){};
\draw[fill] (3, 0) circle (1pt) node(a02){};
\draw (3,-2) node (a22){$*$};
\draw[fill] (4, 0) circle (1pt) node (a03){};
\draw[fill] (4, -1) circle (1pt) node (a13){};
\draw[fill](4, -2) circle (1pt) node (a23){};
\draw (4, -3) node (a33){$*$};
\draw[mono] (a00)--(a01);
\draw[mono] (a11)--(a12);
\draw[mono] (a01)--(a02);
\draw[mono] (a02)--(a03);
\draw[mono] (a12)--(a13);
\draw[mono] (a22)--(a23);
\draw[epi] (a01)--(a11);
\draw[epi] (a12)--(a22);
\draw[epi] (a02)--(a12);
\draw[epi] (a03)--(a13);
 \draw[epi] (a03)..controls (4.2,-1)..(a23);
\draw[epi] (a23)--(a33);
\begin{scope}[yshift=-0.3cm]
  \draw[twoarrowlonger] (2.2,0.1)--(2.8,-0.5);
  \draw[twoarrowlonger] (3.2,0.1)--(3.8,-0.5);
    \draw[twoarrowlonger] (3.2,-0.4)--(3.8,-1.0);
\end{scope}
\end{scope}

\draw (6, -1.3) node[anchor=north east] (j5){$\mapsto$};
\begin{scope}[xshift=6cm, yshift=0cm]
     \draw (1,0) node(a00){$*$};
\draw[fill] (2,0) circle (1pt) node(a01){};
\draw (2,-1) node(a11) {$*$};
\draw[fill] (3, -1) circle (1pt) node(a12){};
\draw[fill] (3, 0) circle (1pt) node(a02){};
\draw (3,-2) node (a22){$*$};
\draw[fill] (4, 0) circle (1pt) node (a03){};
\draw[fill](4, -2) circle (1pt) node (a23){};
\draw (4, -3) node (a33){$*$};
\draw[mono] (a00)--(a01);
\draw[mono] (a11)--(a12);
\draw[mono] (a01)--(a02);
\draw[mono] (a02)--(a03);
\draw[mono] (a22)--(a23);
\draw[epi] (a01)--(a11);
\draw[epi] (a12)--(a22);
\draw[epi] (a02)--(a12);
 \draw[epi] (a03)..controls (4.2,-1)..(a23);
\draw[epi] (a23)--(a33);
\begin{scope}[yshift=-0.3cm]
  \draw[twoarrowlonger] (2.2,0.1)--(2.8,-0.5);
    \draw[twoarrowlonger] (3.2,-0.4)--(3.8,-1.0);
\end{scope}
\end{scope}
    \end{tikzpicture}
\end{center}
and can be understood to be an equivalence of spaces using the property of double Segality and stability for $\cD$.
\end{itemize}
The desired result then follows.
\end{proof}

We can now prove the proposition:

\begin{proof}[Proof of \cref{Sdot2Segal}]
One kind of $2$-Segal map for $S_\bullet\cD$ was discussed as \cref{Lemma2Segal}, and the other one can be treated analogously.
\end{proof}

\begin{rmk}
    There is a functor
    \[
    S_\bullet\colon\mathrm{saDblSegSp}\to\mathrm{2SegSp}\]
\end{rmk}

This functor is compatible with the homotopy theory of the objects involved:

\begin{prop}
\label{Shomotopical}
Let $\cD$ and $\cE$ be injectively fibrant stable augmented double Segal spaces. If $\varphi\colon \cD\xrightarrow{\simeq}\cE$ is an equivalence of stable augmented double Segal spaces, then
\[S_{\bullet}\varphi\colon S_{\bullet}\cD\xrightarrow{\simeq} S_{\bullet}\cE\]
is an equivalence of $2$-Segal spaces.
\end{prop}

The proof of the proposition relies on a technical lemma:

\begin{lem}
\label{FirstRow}
Let $\cD$ be an injectively fibrant
stable augmented double Segal space.
For all $n\geq0$ there is an equivalence of spaces
\[\Map(\cP\Delta[n+1],\cD)\xrightarrow{\simeq}\cD_{0,n}\]
given by essentially by ``selecting the first row''. For instance, when $n=3$ it can be depicted as
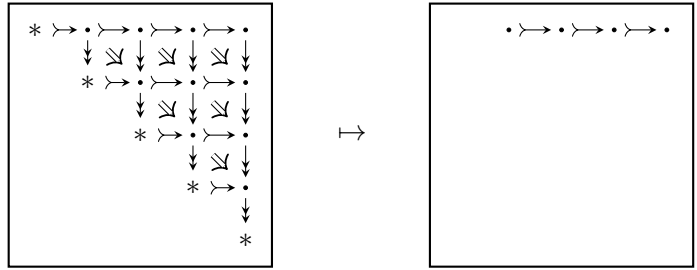
\begin{center} 
\begin{tikzpicture}[scale=0.7]
\draw[thick] (8,0.5) rectangle (13,-4.5);
\begin{scope}[xshift=7.5cm, yshift=0cm]
\draw[fill] (2,0) circle (1pt) node(a01){};
\draw[fill] (3, 0) circle (1pt) node(a02){};
\draw[fill] (4, 0) circle (1pt) node (a03){};
\draw[fill] (5, 0) circle (1pt) node (a04){};
\draw[mono] (a01)--(a02);
\draw[mono] (a02)--(a03);
\draw[mono] (a03)--(a04);
\end{scope}

\draw (7, -1.7) node[anchor=north east] (j5){{$\mapsto$}};
\draw[thick] (0,0.5) rectangle (5,-4.5);
\begin{scope}[xshift=-0.5cm]
     \draw (1,0) node(a00){$*$};
\draw[fill] (2,0) circle (1pt) node(a01){};
\draw (2,-1) node(a11) {$*$};
\draw[fill] (3, -1) circle (1pt) node(a12){};
\draw[fill] (3, 0) circle (1pt) node(a02){};
\draw (3,-2) node (a22){$*$};
\draw[fill] (4, 0) circle (1pt) node (a03){};
\draw[fill] (4, -1) circle (1pt) node (a13){};
\draw[fill](4, -2) circle (1pt) node (a23){};
\draw (4, -3) node (a33){$*$};
\draw[fill] (5, 0) circle (1pt) node (a04){};
\draw[fill] (5, -1)circle (1pt) node (a14){};
\draw[fill] (5, -2)circle (1pt) node (a24){};
\draw[fill] (5, -3)circle(1pt) node (a34){};
\draw (5, -4) node (a44){$*$};
\draw[mono] (a00)--(a01);
\draw[mono] (a11)--(a12);
\draw[mono] (a01)--(a02);
\draw[mono] (a02)--(a03);
\draw[mono] (a12)--(a13);
\draw[mono] (a22)--(a23);
\draw[mono] (a03)--(a04);
\draw[mono] (a13)--(a14);
\draw[mono] (a23)--(a24);
\draw[mono] (a33)--(a34);
\draw[epi] (a01)--(a11);
\draw[epi] (a12)--(a22);
\draw[epi] (a02)--(a12);
\draw[epi] (a03)--(a13);
\draw[epi] (a13)--(a23);
\draw[epi] (a23)--(a33);
\draw[epi] (a04)--(a14);
\draw[epi] (a14)--(a24);
\draw[epi] (a24)--(a34);
\draw[epi] (a34)--(a44);
\begin{scope}[yshift=-0.3cm]
   \draw[twoarrowlonger] (2.2,0.1)--(2.8,-0.5);
   \draw[twoarrowlonger] (3.2,0.1)--(3.8,-0.5);
   \draw[twoarrowlonger] (3.2,-0.9)--(3.8,-1.5);
   \draw[twoarrowlonger] (4.2,0.1)--(4.8,-0.5);
   \draw[twoarrowlonger] (4.2,-1.9)--(4.8,-2.5);
   \draw[twoarrowlonger] (4.2,-0.9)--(4.8,-1.5);
\end{scope}
\end{scope}
\end{tikzpicture}
\end{center}
\end{lem}

We illustrate the idea in the case $n=1$. An analog (though much more intricate) argument can be provided for the case of arbitrary $n$ by adjusting the filtration from \cite[Lemma~5.12]{BOORS3}.

\begin{proof}[Proof Idea of \cref{FirstRow} for $n=1$]
There is a tower of equivalences of spaces of the form:
\[\Map(\cP\Delta[2],\cD)
\xrightarrow{\simeq}F^{(1)}
\xrightarrow{\simeq}F^{(2)}
\xrightarrow{\simeq}\cD_{0,1},\]
which can be constructed as follows.
\begin{itemize}[leftmargin=*]
    \item First, there is an equivalence of spaces
    \[\Map(\cP\Delta[2],\cD)\xrightarrow{\simeq}\cD_{1,1}\times^h_{\cD_{0,0}}\cD_{-1}=:F^{(1)},\]
    which can be depicted as
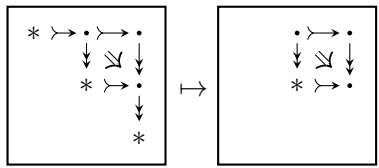
\begin{center}
    \begin{tikzpicture}[scale=0.7]
\draw[thick] (0.5,0.5) rectangle (3.5,-2.5);
\draw[thick] (4.5,0.5) rectangle (7.5,-2.5);
\begin{scope}[xshift=0cm, yshift=0cm]
     \draw (1,0) node(a00){$*$};
\draw[fill] (2,0) circle (1pt) node(a01){};
\draw (2,-1) node(a11) {$*$};
\draw[fill] (3, -1) circle (1pt) node(a12){};
\draw[fill] (3, 0) circle (1pt) node(a02){};
\draw (3,-2) node (a22){$*$};
\draw[mono] (a00)--(a01);
\draw[mono] (a11)--(a12);
\draw[mono] (a01)--(a02);
\draw[epi] (a01)--(a11);
\draw[epi] (a12)--(a22);
\draw[epi] (a02)--(a12);
\begin{scope}[yshift=-0.3cm]
  \draw[twoarrowlonger] (2.2,0.1)--(2.8,-0.5);
\end{scope}
\end{scope}
\draw (4.5, -0.8) node[anchor=north east] (j5){{$\mapsto$}};
%
\begin{scope}[xshift=4cm, yshift=0cm]
\draw[fill] (2,0) circle (1pt) node(a01){};
\draw (2,-1) node(a11) {$*$};
\draw[fill] (3, -1) circle (1pt) node(a12){};
\draw[fill] (3, 0) circle (1pt) node(a02){};
\draw[mono] (a11)--(a12);
\draw[mono] (a01)--(a02);
\draw[epi] (a01)--(a11);
\draw[epi] (a02)--(a12);
\begin{scope}[yshift=-0.3cm]
  \draw[twoarrowlonger] (2.2,0.1)--(2.8,-0.5);
\end{scope}
\end{scope}
\end{tikzpicture}
\end{center}
and can be understood to be an equivalence of spaces using the properties of double Segality and augmentation for $\cD$.
    \item Next, there is an equivalence of spaces of the form
    \[F^{(1)}\xrightarrow{\simeq}\cD_{1,0}\times^h_{\cD_{0,0}}\cD_{0,1}\times^h_{\cD_{0,0}}\cD_{-1}\eqqcolon F^{(2)},\]
   which can be depicted as
    
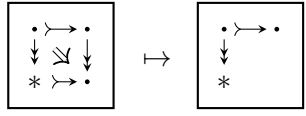
\begin{center}
    \begin{tikzpicture}[scale=0.7]
\draw[thick] (1.5,0.5) rectangle (3.5,-1.5);
\begin{scope}[xshift=3.6cm, yshift=0cm]
\draw[fill] (2,0) circle (1pt) node(a01){};
\draw (2,-1) node(a11) {$*$};
\draw[fill] (3, 0) circle (1pt) node(a02){};
%
\draw[mono] (a01)--(a02);
%
\draw[epi] (a01)--(a11);
\end{scope}
\draw (4.8, -0.3) node[anchor=north east] (j5){{$\mapsto$}};
\draw[thick] (5.1,0.5) rectangle (7.1,-1.5);
\begin{scope}
\draw[fill] (2,0) circle (1pt) node(a01){};
\draw (2,-1) node(a11) {$*$};
\draw[fill] (3, -1) circle (1pt) node(a12){};
\draw[fill] (3, 0) circle (1pt) node(a02){};
\draw[mono] (a11)--(a12);
\draw[mono] (a01)--(a02);
\draw[epi] (a01)--(a11);
\draw[epi] (a02)--(a12);
\begin{scope}[yshift=-0.3cm]
  \draw[twoarrowlonger] (2.2,0.1)--(2.8,-0.5);
\end{scope}
\end{scope}
\end{tikzpicture}
\end{center}
and can be understood to be an equivalence of spaces using the property of Double Segality and stability for $\cD$.
    \item Finally, there is an equivalence of spaces of the form
    \[F^{(2)}\xrightarrow{\simeq}\cD_{0,1},\]
    which can be depicted as
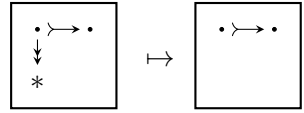
\begin{center}
    \begin{tikzpicture}[scale=0.7]
\draw[thick] (1.5,0.5) rectangle (3.5,-1.5);
\draw[thick] (5,0.5) rectangle (7,-1.5);
\begin{scope}
\draw[fill] (2,0) circle (1pt) node(a01){};
\draw (2,-1) node(a11) {$*$};
\draw[fill] (3, 0) circle (1pt) node(a02){};
%
\draw[mono] (a01)--(a02);
%
\draw[epi] (a01)--(a11);
\end{scope}
\draw (4.8, -0.3) node[anchor=north east] (j5){{$\mapsto$}};
%
\begin{scope}[xshift=3.5cm, yshift=0cm]
\draw[fill] (2,0) circle (1pt) node(a01){};
\draw[fill] (3, 0) circle (1pt) node(a02){};
\draw[mono] (a01)--(a02);
\end{scope}
\end{tikzpicture}
\end{center}
    and can be understood to be an equivalence of spaces using the property of double Segality
    and augmentation for $\cD$.
\end{itemize}
The desired result then follows.
\end{proof}

We can now prove the proposition.

\begin{proof}[Proof of \cref{Shomotopical}]
Since $\varphi\colon\cD\to\cE$ is an equivalence of stable augmented double Segal spaces, for all $n\geq0$ there is an equivalence of spaces
\[
\varphi_{0,n}\colon\cD_{0,n}\xrightarrow{\simeq}\cE_{0,n}.
\]
By \cref{FirstRow}, we can deduce that there is also an equivalence of spaces
\[\varphi_*\colon \Map(\cP\Delta[n],\cD)\xrightarrow{\simeq} \Map(\cP\Delta[n],\cE),\]
which is by definition the map
\[S_{n}\varphi\colon S_{n}\cD\xrightarrow{\simeq} S_{n}\cE.\]
We then obtain an equivalence of $2$-Segal spaces
$S_\bullet\varphi\colon S_{\bullet}\cD\xrightarrow{\simeq} S_{\bullet}\cE$,
as desired.
\end{proof}

Thanks to \cref{Shomotopical}, we obtain that the path construction descends to equivalence classes in the following sense:

\begin{rmk}
    There is a function
    \[
  S_\bullet\circ\widetilde{(-)}\colon\mathrm{saDblSegSp}/_\simeq\ \to\ \mathrm{2SegSp}/_\simeq
  \]
\end{rmk}

Moreover, it follows from \cite[\textsection1.5]{BarwickKan} that the path construction also induces a functor at the level of simplicial localizations, so we obtain:

\begin{const}
The $S_\bullet$-construction induces a simplicial functor
\[
\bm{S_\bullet\circ\widetilde{(-)}}\colon\bm{\mathrm{saDblSegSp}}\to\bm{\mathrm{2SegSp}}.
\]
\end{const}

\section{The equivalence}

We can now explain how the path construction $\cP$ and the $S_\bullet$-construction define inverse correspondences.

In order to show that the path construction and the $S_\bullet$-constructions are inverse to each other, we start by analyzing the first composite:

\begin{prop}
\label{unit}
If $X$ is a $2$-Segal space, then there is an equivalence of $2$-Segal spaces
\[
\eta^h\colon X\xrightarrow{\simeq} S_{\bullet}\widetilde{\cP X}.
\]

\end{prop}

\begin{proof}
There is a canonical simplicial map
\[
\eta^h\colon X\xrightarrow{\eta} S_{\bullet}\cP X\to S_{\bullet}\widetilde{\cP X}.\]
Now, for all $n\geq0$, using the naturality of mapping spaces from \cref{MappingSpaces} we see that there are commutative diagrams of spaces
\[
\begin{tikzcd}
S_{n+1}\cP X\arrow[r]&S_{n+1}\widetilde{\cP X}\arrow[r,"\simeq"]&(\widetilde{\cP X})_{0,n}\\
&X_{n+1}\arrow[r,equal]\arrow[u,"\eta_{n+1}^h" swap]\arrow[lu,"\eta_{n+1}"]&(\cP X)_{0,n}\arrow[u,"\simeq"]
\end{tikzcd}
\text{ and }
\begin{tikzcd}
S_{0}\cP X\arrow[r]&S_{0}\widetilde{\cP X}\arrow[r,"\cong"]&(\widetilde{\cP X})_{-1}\\
&X_{0}\arrow[r,equal]\arrow[u,"\eta_{0}^h" swap]\arrow[lu,"\eta_{0}"]&(\cP X)_{-1}\arrow[u,"\simeq"]
\end{tikzcd}
\]
where the horizontal map in the first diagram is an equivalence as an instance of \cref{FirstRow}.
By two-out-of-three, we then obtain equivalences of spaces
\[
\eta_{n+1}^h\colon X\xrightarrow{\simeq} S_{n+1}\widetilde{\cP X}\quad\text{ and }\quad\eta_{0}^h\colon X_0\xrightarrow{\simeq} S_{0}\widetilde{\cP X},\]
and an equivalence of $2$-Segal spaces
$
\eta^h\colon X\to S_{\bullet}\widetilde{\cP X}$,
as desired.
\end{proof}

Next, we address the other composite:

\begin{prop}
\label{counit}
    If $\cD$ is an injectively fibrant stable augmented double Segal space, then there is an equivalence of stable augmented double Segal spaces
\[
\epsilon\colon \cP S_{\bullet}\cD\xrightarrow{\simeq} \cD.
\]
\end{prop}

The proof of the proposition requires an auxiliary lemma:

\begin{lem}
\label{Filtration11}
Let $\cD$ be an injectively fibrant
stable augmented double Segal space. There is an equivalence of spaces
\[\Map(\cP\Delta[3],\cD)\xrightarrow{\simeq}\cD_{1,1}\]
which can be pictured as
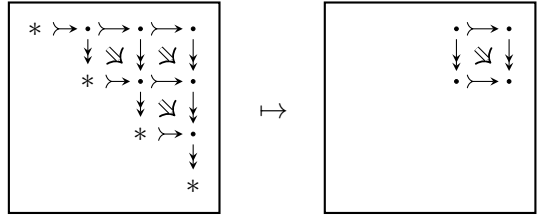
\begin{center}
    \begin{tikzpicture}[scale=0.7]
    \draw[thick] (6.5,0.5) rectangle (10.5,-3.5);
    \draw[thick] (0.5,0.5) rectangle (4.5,-3.5);
\begin{scope}[xshift=6cm, yshift=0cm]
\draw[fill] (3, -1) circle (1pt) node(a12){};
\draw[fill] (3, 0) circle (1pt) node(a02){};
\draw[fill] (4, 0) circle (1pt) node (a03){};
\draw[fill] (4, -1) circle (1pt) node (a13){};

\draw[mono] (a02)--(a03);
\draw[mono] (a12)--(a13);

\draw[epi] (a02)--(a12);
\draw[epi] (a03)--(a13);

\begin{scope}[yshift=-0.3cm]
  \draw[twoarrowlonger] (3.2,0.1)--(3.8,-0.5);
\end{scope}
\end{scope}

\draw (6, -1.3) node[anchor=north east] (j5){$\mapsto$};
\begin{scope}[xshift=0cm, yshift=0cm]
     \draw (1,0) node(a00){$*$};
\draw[fill] (2,0) circle (1pt) node(a01){};
\draw (2,-1) node(a11) {$*$};
\draw[fill] (3, -1) circle (1pt) node(a12){};
\draw[fill] (3, 0) circle (1pt) node(a02){};
\draw (3,-2) node (a22){$*$};
\draw[fill] (4, 0) circle (1pt) node (a03){};
\draw[fill] (4, -1) circle (1pt) node (a13){};
\draw[fill](4, -2) circle (1pt) node (a23){};
\draw (4, -3) node (a33){$*$};
\draw[mono] (a00)--(a01);
\draw[mono] (a11)--(a12);
\draw[mono] (a01)--(a02);
\draw[mono] (a02)--(a03);
\draw[mono] (a12)--(a13);
\draw[mono] (a22)--(a23);
\draw[epi] (a01)--(a11);
\draw[epi] (a12)--(a22);
\draw[epi] (a02)--(a12);
\draw[epi] (a03)--(a13);
\draw[epi] (a13)--(a23);
\draw[epi] (a23)--(a33);
\begin{scope}[yshift=-0.3cm]
  \draw[twoarrowlonger] (2.2,0.1)--(2.8,-0.5);
  \draw[twoarrowlonger] (3.2,0.1)--(3.8,-0.5);
  \draw[twoarrowlonger] (3.2,-0.9)--(3.8,-1.5);
\end{scope}
\end{scope}
    \end{tikzpicture}
\end{center}
\end{lem}

We illustrate the idea, and refer the reader to  \cite[Lemma~6.6]{BOORS3} for a rigorous proof. 

\begin{proof}[Idea of the proof of \cref{Filtration11}]
There is a tower of equivalences of spaces of the form:
\[\Map(\cP\Delta[3],\cD)
\xrightarrow{\simeq}F^{(1)}
\xrightarrow{\simeq}F^{(2)}
\xrightarrow{\simeq}\cD_{1,1},\]
which can be constructed as follows:
\begin{itemize}[leftmargin=*]
    \item First, there is an equivalence of spaces
    \[\Map(\cP\Delta[3],\cD)\xrightarrow{\simeq}\cD_{2,1}\times^h_{\cD_{1,1}}\cD_{1,2}\times^h_{\cD_{0,0}}\cD_{-1}\times^h_{\cD_{0,0}}\cD_{-1}=:F^{(1)},\]
    which can be depicted as
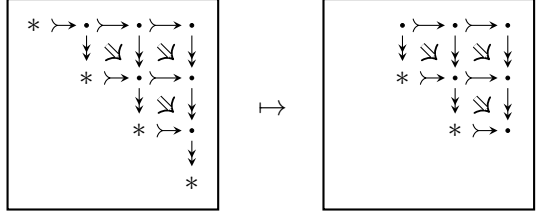
\begin{center}
    \begin{tikzpicture}[scale=0.7]
    \draw[thick] (6.5,0.5) rectangle (10.5,-3.5);
    \draw[thick] (0.5,0.5) rectangle (4.5,-3.5);
\begin{scope}[xshift=6cm, yshift=-0cm]
\draw[fill] (2,0) circle (1pt) node(a01){};
\draw (2,-1) node(a11) {$*$};
\draw[fill] (3, -1) circle (1pt) node(a12){};
\draw[fill] (3, 0) circle (1pt) node(a02){};
\draw (3,-2) node (a22){$*$};
\draw[fill] (4, 0) circle (1pt) node (a03){};
\draw[fill] (4, -1) circle (1pt) node (a13){};
\draw[fill](4, -2) circle (1pt) node (a23){};

\draw[mono] (a11)--(a12);
\draw[mono] (a01)--(a02);
\draw[mono] (a02)--(a03);
\draw[mono] (a12)--(a13);
\draw[mono] (a22)--(a23);

\draw[epi] (a01)--(a11);
\draw[epi] (a12)--(a22);
\draw[epi] (a02)--(a12);
\draw[epi] (a03)--(a13);
\draw[epi] (a13)--(a23);

\begin{scope}[yshift=-0.3cm]
  \draw[twoarrowlonger] (2.2,0.1)--(2.8,-0.5);
  \draw[twoarrowlonger] (3.2,0.1)--(3.8,-0.5);
  \draw[twoarrowlonger] (3.2,-0.9)--(3.8,-1.5);
\end{scope}
\end{scope}

\draw (6, -1.3) node[anchor=north east] (j5){$\mapsto$};
\begin{scope}[xshift=0cm, yshift=0cm]
     \draw (1,0) node(a00){$*$};
\draw[fill] (2,0) circle (1pt) node(a01){};
\draw (2,-1) node(a11) {$*$};
\draw[fill] (3, -1) circle (1pt) node(a12){};
\draw[fill] (3, 0) circle (1pt) node(a02){};
\draw (3,-2) node (a22){$*$};
\draw[fill] (4, 0) circle (1pt) node (a03){};
\draw[fill] (4, -1) circle (1pt) node (a13){};
\draw[fill](4, -2) circle (1pt) node (a23){};
\draw (4, -3) node (a33){$*$};
\draw[mono] (a00)--(a01);
\draw[mono] (a11)--(a12);
\draw[mono] (a01)--(a02);
\draw[mono] (a02)--(a03);
\draw[mono] (a12)--(a13);
\draw[mono] (a22)--(a23);
\draw[epi] (a01)--(a11);
\draw[epi] (a12)--(a22);
\draw[epi] (a02)--(a12);
\draw[epi] (a03)--(a13);
\draw[epi] (a13)--(a23);
\draw[epi] (a23)--(a33);
\begin{scope}[yshift=-0.3cm]
  \draw[twoarrowlonger] (2.2,0.1)--(2.8,-0.5);
  \draw[twoarrowlonger] (3.2,0.1)--(3.8,-0.5);
  \draw[twoarrowlonger] (3.2,-0.9)--(3.8,-1.5);
\end{scope}
\end{scope}
    \end{tikzpicture}
\end{center}
and can be understood to be an equivalence of spaces using the properties of double Segality and augmentation for $\cD$.
    \item Next, there is an equivalence of spaces
    \[F^{(1)}\xrightarrow{\simeq}
    \cD_{1,1}\times^h_{\cD_{0,0}}\cD_{0,1}\times^h_{\cD_{0,0}}\cD_{-1}\times^h_{\cD_{0,0}}\cD_{0,1}\times^h_{\cD_{0,0}}\cD_{-1}=:F^{(2)},\]
    which can be depicted as
   
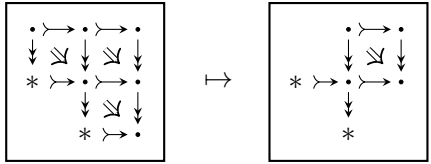
\begin{center}
    \begin{tikzpicture}[scale=0.7]
    \draw[thick] (6.5,0.5) rectangle (9.5,-2.5);
    \draw[thick] (1.5,0.5) rectangle (4.5,-2.5);
\begin{scope}[xshift=5cm, yshift=-0cm]
\draw (2,-1) node(a11) {$*$};
\draw[fill] (3, -1) circle (1pt) node(a12){};
\draw[fill] (3, 0) circle (1pt) node(a02){};
\draw (3,-2) node (a22){$*$};
\draw[fill] (4, 0) circle (1pt) node (a03){};
\draw[fill] (4, -1) circle (1pt) node (a13){};

\draw[mono] (a11)--(a12);
\draw[mono] (a02)--(a03);
\draw[mono] (a12)--(a13);

\draw[epi] (a12)--(a22);
\draw[epi] (a02)--(a12);
\draw[epi] (a03)--(a13);

\begin{scope}[yshift=-0.3cm]
  \draw[twoarrowlonger] (3.2,0.1)--(3.8,-0.5);
\end{scope}
\end{scope}

\draw (6, -0.7) node[anchor=north east] (j5){$\mapsto$};
\begin{scope}[xshift=0cm, yshift=-0cm]
\draw[fill] (2,0) circle (1pt) node(a01){};
\draw (2,-1) node(a11) {$*$};
\draw[fill] (3, -1) circle (1pt) node(a12){};
\draw[fill] (3, 0) circle (1pt) node(a02){};
\draw (3,-2) node (a22){$*$};
\draw[fill] (4, 0) circle (1pt) node (a03){};
\draw[fill] (4, -1) circle (1pt) node (a13){};
\draw[fill](4, -2) circle (1pt) node (a23){};

\draw[mono] (a11)--(a12);
\draw[mono] (a01)--(a02);
\draw[mono] (a02)--(a03);
\draw[mono] (a12)--(a13);
\draw[mono] (a22)--(a23);

\draw[epi] (a01)--(a11);
\draw[epi] (a12)--(a22);
\draw[epi] (a02)--(a12);
\draw[epi] (a03)--(a13);
\draw[epi] (a13)--(a23);

\begin{scope}[yshift=-0.3cm]
  \draw[twoarrowlonger] (2.2,0.1)--(2.8,-0.5);
  \draw[twoarrowlonger] (3.2,0.1)--(3.8,-0.5);
 \draw[twoarrowlonger] (3.2,-0.9)--(3.8,-1.5);
\end{scope}
\end{scope}
    \end{tikzpicture}
\end{center}
and can be understood to be an equivalence of spaces using the property of double Segality and stability for $\cD$.
    \item Finally, there is an equivalence of spaces
    \[F^{(2)}\xrightarrow{\simeq}\cD_{1,1},\]
    which can be depicted as

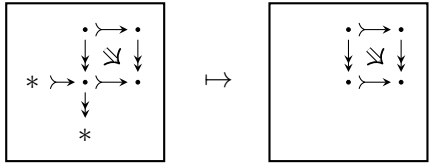
\begin{center}
    \begin{tikzpicture}[scale=0.7]
    \draw[thick] (6.5,0.5) rectangle (9.5,-2.5);
    \draw[thick] (1.5,0.5) rectangle (4.5,-2.5);
\begin{scope}[xshift=0cm, yshift=-0cm]
\draw (2,-1) node(a11) {$*$};
\draw[fill] (3, -1) circle (1pt) node(a12){};
\draw[fill] (3, 0) circle (1pt) node(a02){};
\draw (3,-2) node (a22){$*$};
\draw[fill] (4, 0) circle (1pt) node (a03){};
\draw[fill] (4, -1) circle (1pt) node (a13){};

\draw[mono] (a11)--(a12);
\draw[mono] (a02)--(a03);
\draw[mono] (a12)--(a13);

\draw[epi] (a12)--(a22);
\draw[epi] (a02)--(a12);
\draw[epi] (a03)--(a13);

\begin{scope}[yshift=-0.3cm]
  \draw[twoarrowlonger] (3.2,0.1)--(3.8,-0.5);
\end{scope}
\end{scope}

\draw (6, -0.7) node[anchor=north east] (j5){$\mapsto$};
\begin{scope}[xshift=5cm, yshift=-0cm]
\draw[fill] (3, -1) circle (1pt) node(a12){};
\draw[fill] (3, 0) circle (1pt) node(a02){};
\draw[fill] (4, 0) circle (1pt) node (a03){};
\draw[fill] (4, -1) circle (1pt) node (a13){};

\draw[mono] (a02)--(a03);
\draw[mono] (a12)--(a13);

\draw[epi] (a02)--(a12);
\draw[epi] (a03)--(a13);

\begin{scope}[yshift=-0.3cm]
  \draw[twoarrowlonger] (3.2,0.1)--(3.8,-0.5);
\end{scope}
\end{scope}
    \end{tikzpicture}
\end{center}
and can be understood to be an equivalence of spaces using the property of double Segality and augmentation for $\cD$.
\end{itemize}
The desired result then follows.
\end{proof}

We can now prove the proposition:

\begin{proof}[Proof of \cref{counit}]
For all $a,b\geq0$, we consider the map of spaces
\[
\cP_{a,b}(S_\bullet\cD)=S_{a+1+b}\cD\to\cD_{a,b}\]
which ``selects the internal grid of size $[a]\times[b]$''. For instance, when $a=1$ and $b=2$ it can be depicted as
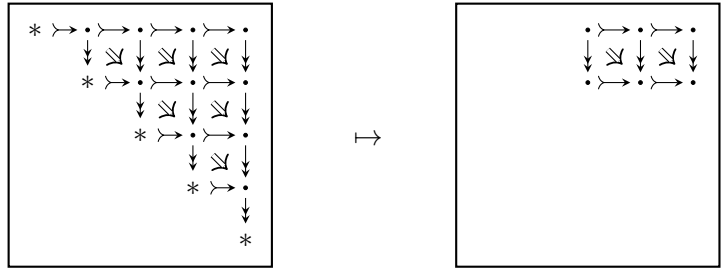
\begin{center} 
\begin{tikzpicture}[scale=0.7]

\draw (7.3, -1.8) node[anchor=north east] (j5){{$\mapsto$}};
 \draw[thick] (8.5,0.5) rectangle (13.5,-4.5);
\begin{scope}[xshift=9cm, yshift=-0cm]
\draw[fill] (2,0) circle (1pt) node(a01){};
\draw[fill] (2,-1) circle (1pt) node(a11){};
\draw[fill] (3, -1) circle (1pt) node(a12){};
\draw[fill] (3, 0) circle (1pt) node(a02){};
\draw[fill] (4, 0) circle (1pt) node (a03){};
\draw[fill] (4, -1) circle (1pt) node (a13){};

\draw[mono] (a11)--(a12);
\draw[mono] (a01)--(a02);
\draw[mono] (a02)--(a03);
\draw[mono] (a12)--(a13);

\draw[epi] (a01)--(a11);
\draw[epi] (a02)--(a12);
\draw[epi] (a03)--(a13);

\begin{scope}[yshift=-0.3cm]
  \draw[twoarrowlonger] (2.2,0.1)--(2.8,-0.5);
  \draw[twoarrowlonger] (3.2,0.1)--(3.8,-0.5);
\end{scope}
\end{scope}
\draw[thick] (0,0.5) rectangle (5,-4.5);
\begin{scope}[xshift=-0.5cm]
     \draw (1,0) node(a00){$*$};
\draw[fill] (2,0) circle (1pt) node(a01){};
\draw (2,-1) node(a11) {$*$};
\draw[fill] (3, -1) circle (1pt) node(a12){};
\draw[fill] (3, 0) circle (1pt) node(a02){};
\draw (3,-2) node (a22){$*$};
\draw[fill] (4, 0) circle (1pt) node (a03){};
\draw[fill] (4, -1) circle (1pt) node (a13){};
\draw[fill](4, -2) circle (1pt) node (a23){};
\draw (4, -3) node (a33){$*$};
\draw[fill] (5, 0) circle (1pt) node (a04){};
\draw[fill] (5, -1)circle (1pt) node (a14){};
\draw[fill] (5, -2)circle (1pt) node (a24){};
\draw[fill] (5, -3)circle(1pt) node (a34){};
\draw (5, -4) node (a44){$*$};
\draw[mono] (a00)--(a01);
\draw[mono] (a11)--(a12);
\draw[mono] (a01)--(a02);
\draw[mono] (a02)--(a03);
\draw[mono] (a12)--(a13);
\draw[mono] (a22)--(a23);
\draw[mono] (a03)--(a04);
\draw[mono] (a13)--(a14);
\draw[mono] (a23)--(a24);
\draw[mono] (a33)--(a34);
\draw[epi] (a01)--(a11);
\draw[epi] (a12)--(a22);
\draw[epi] (a02)--(a12);
\draw[epi] (a03)--(a13);
\draw[epi] (a13)--(a23);
\draw[epi] (a23)--(a33);
\draw[epi] (a04)--(a14);
\draw[epi] (a14)--(a24);
\draw[epi] (a24)--(a34);
\draw[epi] (a34)--(a44);
\begin{scope}[yshift=-0.3cm]
   \draw[twoarrowlonger] (2.2,0.1)--(2.8,-0.5);
   \draw[twoarrowlonger] (3.2,0.1)--(3.8,-0.5);
   \draw[twoarrowlonger] (3.2,-0.9)--(3.8,-1.5);
   \draw[twoarrowlonger] (4.2,0.1)--(4.8,-0.5);
   \draw[twoarrowlonger] (4.2,-1.9)--(4.8,-2.5);
   \draw[twoarrowlonger] (4.2,-0.9)--(4.8,-1.5);
\end{scope}
\end{scope}
\end{tikzpicture}
\end{center}

This map can be checked to be natural in $a,b\geq0$, and it hence defines a map
of preaugmented bisimplicial spaces
\[
\epsilon\colon \cP S_{\bullet}\cD\to\cD.
\]
By \cref{Sdot2Segal,PathSADSS}, both $\cP S_{\bullet}\cD$ and $\cD$ are stable augmented double Segal spaces, so $\epsilon$ is a map of stable augmented double Segal spaces.
In order to show that it is an equivalence of stable augmented double Segal spaces, it suffices to check that it induces equivalences of spaces $\epsilon_{-1}$, $\epsilon_{0,0}$, $\epsilon_{0,1}$, $\epsilon_{1,0}$, and $\epsilon_{1,1}$, which we now analyze.

The map $\epsilon_{\cD,-1}$ is an isomorphism of spaces
\[\epsilon_{-1}\colon(\cP S_{\bullet}\cD)_{-1}= S_{0}\cD=\Map(\cP\Delta[0],\cD)\cong \cD_{-1}.\]
The map $\epsilon_{0,0}$ can be identified with an instance of the equivalence of spaces from \cref{FirstRow}:
\[\epsilon_{0,0}\colon(\cP S_{\bullet}\cD)_{0,0}=S_{1}\cD
    =\Map(\cP\Delta[1],\cD)
    \simeq \cD_{0,0}.\]
The map $\epsilon_{0,1}$, resp.~$\epsilon_{1,0}$), can be identified with an instance of the equivalence of spaces from \cref{FirstRow}, resp.~of an appropriate dual statement of \cref{FirstRow}:
\[
    \epsilon_{1,0}\colon(\cP S_{\bullet}\cD)_{1,0}=S_{2}\cD\cong\Map(\cP\Delta[2],\cD)\simeq \cD_{1,0},\]
resp.~
\[\epsilon_{0,1}\colon (\cP S_{\bullet}\cD)_{0,1}=S_{2}\cD=\Map(\cP\Delta[2],\cD)\simeq \cD_{0,1}.\]
Finally $\epsilon_{1,1}$ can be identified with an instance of the equivalence of spaces from \cref{Filtration11}:
\[\epsilon_{1,1}\colon(\cP S_{\bullet}\cD)_{1,1}=S_{1}\cD\cong\Map(\cP\Delta[3],\cD)\simeq \cD_{1,1}.\]
This concludes the proof.
\end{proof}

We can finally establish the desired correspondence:

\begin{thm}
\label{EquivalenceV1}
The path construction and $S_\bullet$-construction -- modulo injectively fibrant replacement -- induce inverse bijections:
    \[
    \cP\colon\mathrm{2SegSp}/_\simeq\ \cong\ \mathrm{saDblSegSp}/_\simeq\colon S_\bullet\circ\widetilde{(-)}.
    \]
\end{thm}

\begin{proof}
The statement follows from \cref{unit,counit}.
\end{proof}

The stronger version of the statement, which is the main result from \cite{BOORS3}, asserts that the correspondence from \cref{EquivalenceV1} can actually be upgraded to an equivalence of $\infty$-categories. In this context, an equivalence of $\infty$-categories can be understood as an equivalence of simplicial categories. Recall from \cite[\textsection2.4]{DwyerKanFunction} that an equivalence of simplicial categories (a.k.a.~a \emph{Dwyer--Kan equivalence}) is a simplicial functor which is essentially surjective in an appropriate sense and induces equivalences at the level of mapping spaces.

\begin{thm}[{\cite{BOORS3}}]
\label{EquivalenceV2}
The path construction and the $S_\bullet$-construction (modulo injectively fibrant replacement) induce homotopy inverse Dwyer-Kan equivalences:
\[
\bm{\cP}\colon\bm{\mathrm{2SegSp}}\simeq\bm{\mathrm{saDblSegSp}}\colon \bm{S_\bullet\circ\widetilde{(-)}}.
\]
\end{thm}

We briefly comment on the machinery used to prove the theorem. We refer the reader to standard model categorical references (\cite{Hirschhorn,Hovey}) for the model categorical terminology.

\begin{proof}[Proof strategy]
It is described in \cite[\textsection4.2]{BOORS3} (resp.~\cite[\textsection5.3]{DKbook}) how to put model structures on the category of simplicial spaces (resp.~preaugmented bisimplicial spaces) in which the fibrant-cofibrant objects are the injectively fibrant $2$-Segal spaces (resp.~injectively fibrant stable augmented double Segal spaces). The underlying $\infty$-category is then equivalent as a simplicial category to the one mentioned in \cref{2SegSp} (resp.~\cref{saDblSegSp}). It is shown as \cite[Theorem~6.1]{BOORS3} that the path construction $\cP$ and the $S_\bullet$-construction form a Quillen equivalence. Using the results from \cite{MazelGee}, one then deduces that the path construction $\cP$ and the $S_\bullet$-construction induce inverse Dwyer-Kan equivalences of simplicial categories, as desired.
\end{proof}

\bibliographystyle{amsalpha}
\bibliography{ref}

\end{document}